\documentclass[11pt]{amsart}

\usepackage{amsthm} 
\usepackage{amsbsy}
\usepackage{amssymb}
\usepackage{amsfonts}
\usepackage{amsmath}
\usepackage{verbatim}     
\usepackage{tikz}
\usetikzlibrary{calc,fadings,decorations.pathreplacing}
\usepackage{verbatim}   
\usepackage{graphicx}

\theoremstyle{plain}     
\newtheorem{thm}{Theorem}[section]  
    
\newtheorem{prop}[thm]{Proposition}      
\newtheorem{lemma}[thm]{Lemma}  
\newtheorem{coro}[thm]{Corollary}

\long\def\symbolfootnote[#1]#2{\begingroup% 
\def\thefootnote{\fnsymbol{footnote}}\footnote[#1]{#2}\endgroup}   
 
\theoremstyle{remark}
\newtheorem{rmk}[thm]{Remark}

\title{Arithmetic, Zeros, and Nodal Domains on the Sphere} 
\author{Michael Magee} 
\address{Department of Mathematics, University of California at Santa Cruz, 1156 High Street, Santa Cruz, CA 95064.}
\email{mmagee@ucsc.edu} 
    
\newcommand\pgfmathsinandcos[3]{%
  \pgfmathsetmacro#1{sin(#3)}%
  \pgfmathsetmacro#2{cos(#3)}%
}
\newcommand\LongitudePlane[3][current plane]{%
  \pgfmathsinandcos\sinEl\cosEl{#2} % elevation
  \pgfmathsinandcos\sint\cost{#3} % azimuth
  \tikzset{#1/.estyle={cm={\cost,\sint*\sinEl,0,\cosEl,(0,0)}}}
}
\newcommand\LatitudePlane[3][current plane]{%
  \pgfmathsinandcos\sinEl\cosEl{#2} % elevation
  \pgfmathsinandcos\sint\cost{#3} % latitude
  \pgfmathsetmacro\yshift{\cosEl*\sint}
  \tikzset{#1/.estyle={cm={\cost,0,0,\cost*\sinEl,(0,\yshift)}}} %
}
\newcommand\DrawLongitudeCircle[2][1]{
  \LongitudePlane{\angEl}{#2}
  \tikzset{current plane/.prefix style={scale=#1}}
  \pgfmathsetmacro\angVis{atan(sin(#2)*cos(\angEl)/sin(\angEl))} %
  \draw[current plane] (\angVis:1) arc (\angVis:\angVis+180:1);
  \draw[current plane,dashed] (\angVis-180:1) arc (\angVis-180:\angVis:1);
}
\newcommand\DrawLatitudeCircle[2][1]{
  \LatitudePlane{\angEl}{#2}
  \tikzset{current plane/.prefix style={scale=#1}}
  \pgfmathsetmacro\sinVis{sin(#2)/cos(#2)*sin(\angEl)/cos(\angEl)}
  \pgfmathsetmacro\angVis{asin(min(1,max(\sinVis,-1)))}
  \draw[current plane] (\angVis:1) arc (\angVis:-\angVis-180:1);
  \draw[current plane,dashed] (180-\angVis:1) arc (180-\angVis:\angVis:1);
}

\tikzset{%
  >=latex, % option for nice arrows
  inner sep=0pt,%
  outer sep=2pt,%
  mark coordinate/.style={inner sep=0pt,outer sep=0pt,minimum size=3pt,
    fill=black,circle}%
}
\begin{document}

\maketitle

\def\R{\mathbf{R}}
\def\C{\mathbf{C}}
\def\Z{\mathbf{Z}} 
\def\Q{\mathbf{Q}}

\def\a{\alpha}
\def\b{\beta}
\def\g{\gamma}
\def\d{\delta}
\def\e{\epsilon}
\def\i{\iota}
\def\G{\Gamma}
\def\GG{\mathbf{G}}
\def\vp{\varphi}

\def\GL{\mathrm{GL}}

\def\A{\mathcal{A}} % algebra 
\def\AA{\mathbb{A}} % adeles
\def\B{\mathcal{B}} % basis
\def\E{\mathbb{E}}
\def\H{\mathcal{H}} % space of harmonics
\def\N{\mathcal{N}}
\def\O{\mathcal{O}} % order
\def\Ohat{\widehat{\O}} % closure of order
\def\P{\mathcal{P}}
 \def\k{\kappa}
\def\T{\mathbb{T}}
\begin{abstract}
We obtain lower bounds for the number of nodal domains of Hecke eigenfunctions on the sphere. Assuming the generalized Lindel\"{o}f hypothesis we prove that the number of nodal domains of any Hecke eigenfunction grows with the eigenvalue of the Laplacian. By a very different method, we show unconditionally that the average number of nodal domains of degree $l$ Hecke eigenfunctions grows significantly faster than the uniform growth obtained under Lindel\"{o}f.
\end{abstract}

%%%%%%%%%%%%%%%%%%%%%%%%%%%%%%%%%%%%%%%%%%%%%%%%%%%%%%%%%%%%%%%%%%%%%%%%%%%%%%%%%%%%%%%%%%%%%%%%%%%
%%%%%%%%%%%%%%%%%%%%%%%%%%%%%%%%%%%%%%%%%%%%% INTRODUCTION %%%%%%%%%%%%%%%%%%%%%%%%%%%%%%%%%%%%%%%%
%%%%%%%%%%%%%%%%%%%%%%%%%%%%%%%%%%%%%%%%%%%%%%%%%%%%%%%%%%%%%%%%%%%%%%%%%%%%%%%%%%%%%%%%%%%%%%%%%%%

\section{Introduction}

Let $S^2$ denote the 2-dimensional unit sphere with the round metric and associated measure. The Hilbert space $L^2(S^2)$ splits into a completed tensor sum
\begin{equation}
L^2(S^2) = \overline{\bigoplus}_l \H_l
\end{equation}
where $\H_l$ is the $(2l + 1)$-dimensional space of functions on $S^2$ obtained by restricting degree $l$ homogenenous harmonic polynomials on $\R^3$. The space $\H_l$ coincides with the eigenspace of the spherical Laplacian $\Delta = \Delta_{S^2}$ with eigenvalue $l(l+1)$. For $f \in  \H_l$ the nodal set of $f$ is defined to be
\begin{equation}
Z(f) = \{ v \in S^2 \: : \: f(v) = 0 \} = f^{-1}(0).
\end{equation}
A nodal domain of $f$ is a connected component of $S^2 - Z(f)$. We write $\N(f)$ for the number of nodal domains of $f$, we are interested in estimating $\N(f)$ as $f$ varies. The nodal domain theorem of Courant \cite[pg. 452]{CH} gives the upper bound
\begin{equation}
\N(f) \leq (l+1)^2
\end{equation}
for any $f \in \H_l$. Nontrivial lower bounds for arbitrary spherical harmonics are not possible: a result of Lewy \cite{LEWY} states that there are infinitely many $l$ and $f \in \H_l$ such that $\N(f) \leq 3$. Therefore to obtain growth of $\N(f)$, one has to take $f$ from proper subsets of $\H_l$. We will take $f$ to be eigenfunctions of certain Hecke operators arising from a maximal order in the Hamilton quaternions which we explain now.

Let $B$ be the normed involutive $\Q$-algebra of Hamilton quaternions and $\O$ a particular maximal order in $B(\Q)$ given by
\begin{equation}
\O = \left\{ \frac{ \a + \b \mathbf{i} + \g \mathbf{j} + \delta \mathbf{k} }{2} \: : \: \a , \b , \g , \d \in \Z ,\: \a \equiv \b \equiv \g \equiv \d \bmod 2 \right\}.
\end{equation}
We write for the units of $\O$ 
\begin{equation}
\O^{\times} = \langle 1 , \mathbf{i} , \mathbf{j} , \frac{1}{2}(1 + \mathbf{i} + \mathbf{j} + \mathbf{k} )   \rangle
\end{equation}
which forms a group of order 24.  One can view the sphere $S^2$ as the elements of $B(\R)$ with trace zero and norm one, and there is a corresponding action of $B(\R)$ on $S^2$ given by
\begin{equation}
\g . x \equiv \g x \bar{\g} / n(\g), \quad \g \in B(\R) , \: x \in S^2.
\end{equation} 
The algebra $B$ gives rise to self-adjoint Hecke operators on $L^2(S^2)$. Letting
\begin{equation}
\O(m) \equiv \{ \g \in \O \: : \: n(\g) = m \}
\end{equation}
we define
\begin{equation}\label{eq:hecke}
T_m : L^2(S^2) \to L^2(S^2) , \quad [T_m f](x) = \sum_{\g \in \O(m)} f( \g . x ) .
\end{equation}
The action of these operators was originally described by Eichler in terms of Brandt matrices in \cite{EICH}. \textit{Throughout this paper we make the assumption that any $m$ appearing as the parameter of a Hecke operator $T_m$ is odd to avoid complications at $2$, the discriminant of the Hamilton quaternions.}
It is a classical fact of number theory that the Hecke operators $T_m$ commute and satisfy the recursion
\begin{equation}\label{eq:recursion}
T_m T_n = | \O^\times | \sum_{d | (m,n)} d \: T_{mn/d^2}.
\end{equation}
 The Laplacian $\Delta$ commutes with all rotations in $\mathrm{SO}(3)$, and as the Hecke operators $T_m$ are constituted by such rotations, all the $T_m$ commute with $\Delta$. It follows that the $T_m$ preserve each of the finite dimensional vector spaces $\H_l$. As such, the spaces $\H_l$ each have a basis of simultaneous eigenfunctions $\vp$ for the Hecke operators $T_m$ and the spherical Laplacian $\Delta$. The real valued functions on $S^2$ are $\Delta$ and $T_m$ invariant, so the $\vp$ can be taken to be real valued. These $\vp$ are the Hecke eigenfunctions in question and were considered also in \cite{BSSP} in connection with mass equidistribution conjectures.

Such $\vp$ are important objects in the theory of automorphic forms: they parameterize classical holomorphic cuspidal newforms of level 2. The correspondence is via lifting and theta series with harmonic polynomial on $\R^4$ and concretely given by
\begin{equation}
\vp \in \H_l^{\O^{\times}} \mapsto f_\vp \equiv \frac{1}{24} \sum_{\g \in \O} \langle \vp(x), \vp(\g x \bar{\g}) \rangle_{L^2(S^2)} \exp (2 \pi i n(\g) z) ,
\end{equation}
where $\vp( \g x \bar{\g} )$ is evaluated by extension of $\vp$ to a polynomial on $\R^3$.

The resulting $f_\vp$ is a holomorphic cusp form of weight $k(l) = 2l + 2$ for $\G_0(2)$. The cusp form $f_\vp$ is a Hecke eigenform when $\vp$ is, and $f_\vp$ is Hecke normalized if $\| \vp \|_{L^2(S^2)} = 1$. Eichler shows in \cite{EICH} that all newforms are obtained in this manner, or by linear combinations, and more generally what is being discussed is a concrete realization of the Jacquet-Langlands correspondence. If $T_m \vp = \lambda_\vp(m) \vp$  and $\vp \in \H_l^{\O^{\times}}$ is $L^2$-normalized then
\begin{equation}
f_\vp = \frac{1}{24} \sum_{n \geq 1} \lambda_\vp(n) n^{l} \exp( 2 \pi i n z ) 
\end{equation}
links the eigenvalues of the Hecke operators on $\H_l^{\O^{\times}}$ to the coefficients of Hecke eigenforms in $S^{\mathrm{new}}_{2l+2}(\G_0(2))$. Via the previously outlined correspondence and known estimates for the dimension of the space of cuspidal newforms one has
\begin{equation}\label{eq:dimension}
\dim( \H_l^{\O^{\times}} ) = \frac{l}{6} + O(1).
\end{equation}

A central heuristic of modern number theory is that automorphic data should be simulated by random processes. The comparison random model in our setting is the Random Wave Model (RWM). In the RWM, the $L^2$-normalized random spherical harmonic in $\H_l$ is given by
\begin{equation}
\frac{ \sum_{i = 0}^{2l} \eta_i \phi_i }{ \sqrt{  \sum_{i = 0}^{2l} \eta_i^2 }},
\end{equation}
where the $\phi_i$ form an orthonormal basis for $\H_l$ and the $\eta_i$ are i.i.d. mean zero Gaussian random variables. This random variable does not depend on the choice of $\phi_i$ and we omit the event $\eta_i \equiv 0$ so that the above ratio makes sense.
 Via a Waldspurger type period formula which we explain in Section \ref{periodsection}, it is compatible with the conjectures made by Conrey and Farmer in \cite{CF} that the RWM shares statistics with Hecke eigenfunctions. The conjectures in \cite{CF} are part of a program initiated by Katz and Sarnak in \cite{KS}, which associates symmetry types to families of $L$-functions and makes predictions about the statistics of the $L$-functions based on their symmetry type and possibly also data coming from the functional equation of the $L$-function. The Conrey-Farmer conjectures predict asymptotics for the moments of central values of automorphic $L$-functions as the automorphic representation runs through a family - the period formula we will use translates the Conrey-Farmer conjectures into predictions about the moments of Fourier coefficients of Hecke eigenfunctions.

 One has then the following two rough, conjectural heuristics:
\begin{description}
\item[Ensemble behavior] In the limit as $l \to \infty$, the statistics of the measurement $\N(f)$ on the ensemble
\begin{equation}
\{  f \in \H_l : \| f \|_{L^2(S^2)} = 1 , \: f \text{ a Hecke eigenfunction} \}
\end{equation}
tends to that of $\N(f)$ sampled on random waves in $\H_l$.\
\item [Uniform behaviour] All the Hecke eigenfunctions $f$ in $\H_l$ should have $\N(f)$ like that of `typical' random waves.
\end{description}

Blum, Gnutzmann and Smilansky have argued in \cite{BGS} with numerical evidence that the nodal domain statistics of wavefunctions can identify quantum chaos when present. Around the same time, using a Potts model for a random percolation process, Bogomolny and Schmit \cite{BS} made precise predictions about the nodal domain statistics of both chaotic and random wavefunctions. Bogomolny and Schmit conjecture that for random/chaotic waves in $\H_l$ the expectation and variance of $\N(f)$ are
\begin{align}\label{eq:expectationprediction}
\E \N(f)  &= (a  + o_{l \to \infty}(1)) l^2  , \\
\label{eq:varianceprediction}\mathrm{Var} \N (f) &= (b  + o_{l \to \infty}(1)) l^2,
\end{align}
for some positive $a, b$. In \cite{NS}, Nazarov and Sodin verify the prediction \eqref{eq:expectationprediction} for the RWM and further show that the distribution of $\N(f) / l^2$ concentrates exponentially around $a$. One can therefore reasonably conjecture that the nodal domain statistics of Hecke eigenfunctions are as in equations \eqref{eq:expectationprediction} and \eqref{eq:varianceprediction}.

It bears mentioning at this point that one must take some caution with the heuristics we have outlined. While everything mentioned so far is expected to be true, one can not push conjectures about Hecke eigenfunctions behaving uniformly like typical random waves too far. By adapting a method of Soundararajan from \cite{SOUND}, Mili\'{c}evi\'{c} has shown in \cite{MIL} that on arithmetic Riemann surfaces the Hecke eigenfunctions take on values significantly larger than predicted by the RWM for these surfaces. It is likely that this proof goes through on the sphere to give a similar result.

Having briefly mentioned the connection with randomness we turn to the connection with Quantum chaos. There are direct connections between randomness and Quantum chaos which we will not dwell on, simply referring the reader to the paper of Berry \cite{BERRY}, where it is argued that wave functions of classically chaotic systems are well described by Gaussian random waves.

We focus instead on the connection between Hecke eigenfunctions and Quantum chaos. The operator 
\begin{equation}
S = \frac{\hbar^2}{2} \Delta
\end{equation}
is the Schr\"{o}dinger operator for free mass one dynamics on the sphere. The stationary Schr\"{o}dinger equation for a wavefunction $\psi$ with energy $E$ is then
\begin{equation}
\frac{\hbar^2}{2} \Delta \psi = E \psi ,
\end{equation}
which can be rewritten
\begin{equation}
\Delta \psi = \frac{2E}{\hbar^2 }\psi.
\end{equation}
Wavefunctions in the semiclassical limit $(\hbar \to 0)$ having energies in some band bounded away from 0 therefore correspond to eigenfunctions of the spherical Laplacian $\Delta$ in $\H_l$ with eigenvalues $l(l+1) \to \infty$. It is important to bear in mind that the classical free (geodesic) dynamics on the sphere is integrable - this is a major difference between other arithmetic considerations of Quantum chaos, for example on Shimura curves where the geodesic flow is Anosov. Indeed, there are highest weight vectors $\phi \in \H_l$ with $l \to \infty$ which concentrate their mass on geodesics, so one sees the integrability and periodic orbits of the classical system in the semiclassical limit. 

For each proposed signature of Quantum chaos, there is a corresponding family of problems in number theory where one tries to establish the signature behaviour for various automorphic forms. For an overview of the subject we refer the reader to lecture notes of Sarnak \cite{SARNAKLN}. We will not give any overview here, but will mention the result of VanderKam \cite{VAND} on the $L^\infty$ norms of Hecke eigenfunctions on the sphere as we will use it later. For reference, eigenfunctions $\phi$ with $\Delta \phi = \lambda \phi$ corresponding to classically chaotic systems (for example, on Riemann surfaces) are expected to have
\begin{equation}
\frac{\| \phi \|_\infty}{\| \phi \|_2} \ll_\e \lambda^\e ,
\end{equation}
which is much stronger than the well known bound appearing in \cite{SS} which says that for any eigenfunction of the Laplacian with eigenvalue $\lambda$ on a compact manifold of dimension $n$
\begin{equation}\label{eq:trivialinfinity}
\frac{\| \phi \|_\infty}{\| \phi \|_2} \ll \lambda^{(n-1)/4}.
\end{equation}
In \cite{IS}, Iwaniec and Sarnak obtained an improvement in the exponent over \eqref{eq:trivialinfinity} when $\phi$ is an eigenfunction of the Laplacian and all the Hecke operators on certain arithmetic Riemann surfaces. By adapting their method to the sphere, VanderKam obtained the same improvement in the exponent for Hecke eigenfunctions on the sphere in \cite{VAND}. VanderKam considers slightly different Hecke operators from ours, but the proof goes through mutatis mutandis in our setting and gives
\begin{thm}[VanderKam]\label{vkinfinity}
For any $\e > 0$
\begin{equation}
\frac{\| \vp \|_{L^\infty(S^2)}}{\| \vp \|_{L^2(S^2)}} \ll_\e l^{5/12 + \e}
\end{equation}
when $\vp \in \H_l^{\O^\times}$ is a Hecke eigenfunction.
\end{thm}

In \cite{GRS}, Ghosh, Reznikov and Sarnak give a lower bound for the number of nodal domains of a Maass form on the modular curve. The bound is conditional on the Lindel\"{o}f hypothesis for the $L$-function associated to a Maass form and relies crucially on the $L^\infty$ bound from \cite{IS}. 
\begin{thm}[Ghosh, Reznikov, Sarnak]\label{GRStheorem} Let $\phi$ denote a Maass form for the modular curve $X = \mathrm{SL}_2(\Z) \backslash \mathbb{H}$, i.e. a real valued function on $X$ satisfying
\begin{equation}
\Delta_X \phi = \lambda \phi , \quad \lambda > 0.
\end{equation}
Assuming the Lindel\"{o}f hypothesis for $L(1/2 + it, \phi)$, for any fixed $\e > 0$ and any such $\phi$
\begin{equation}
\N(\phi) \gg_\e \lambda^{1/24 - \e}.
\end{equation}
\end{thm}
We will say more about the proof of Theorem \ref{GRStheorem} shortly as our Theorem \ref{uniformzeros} is completely analogous and both our results share ideas with the paper \cite{GRS}. 

Recalling that $\N(\varphi)$ is the number of nodal domains of a spherical harmonic $\varphi$, we now state the main Theorems of this paper.

\begin{thm} \label{totalzeros}
Let $l$ be even and $\B_l$ denote an orthonormal basis for $\H^{\O^\times}_l$ consisting of Hecke eigenfunctions. Then for any $\e > 0$ we have the lower bound for the total number of nodal domains amongst members of $\B_l$
\begin{equation}
\sum_{\varphi \in \B_l}  \N(\varphi)  \gg_\e l^{5/4 - \e}.
\end{equation}
\end{thm}

\begin{coro}[The number of nodal domains grows on average]\label{zeroscoro}
With everything as in Theorem \ref{totalzeros} we have for the expected number of nodal domains amongst members of $\B_l$, for any $\e > 0$
\begin{equation}
\mathbb{E}_{\B_l}( \N(\varphi) ) = \frac{\sum_{\varphi \in \B_l}  \N(\varphi) }{| \B_l |}  \gg_\e l^{1/4 - \e}.
\end{equation}
\end{coro}
 
\begin{thm}[The number of nodal domains grows uniformly, assuming Lindel\"{o}f]\label{uniformzeros}
Assuming the generalized Lindel\"{o}f hypothesis, for any $\e > 0$ we have
\begin{equation}
\N(\vp) \gg_\e l^{1/12 - \e}
\end{equation} 
when $l$ is even and $\vp \in \H_l^{\O^\times}$ is a Hecke eigenfunction.
\end{thm}

Before sketching the proofs of these Theorems, let us make some general remarks. The proofs of Theorem \ref{totalzeros} and Theorem \ref{uniformzeros} use the same topological method to find nodal domains. This idea appears in \cite{GRS} and enables one, by use of Euler's formula, to obtain nodal domains of eigenfunctions from \textit{isolated} zeros on certain geodesic segments. This is the reason for the requirement that $l$ be even. The restriction of a spherical harmonic in $\H_l$ to an equator is a polynomial of degree $l$. The group of elements in $O(3)$ that commute with the Hecke operators is larger than $\O^\times$, and when $l$ is even  this forces the polynomial to be not identically zero on certain convenient geodesics. The details of this argument appear in Section \ref{euler}. In fact, it should possible to remove the requirement that $l$ is even, following the argument of \cite[Appendix C]{GRS}, but for simplicity we deal only with even $l$.

This approach has some interesting consequences. In the best case scenario, one obtains for $\vp \in \H_l^{\O^\times}$ only linearly as many nodal domains as zeros of a degree $l$ polynomial, in particular $\ll l$ nodal domains. This means that if one expects the prediction of Bogomolny and Schmit \eqref{eq:expectationprediction} to hold, then one misses most of the nodal domains with this approach; it will not yield a nodal domain whose boundary does not meet the geodesic in question. On the other hand, in establishing \eqref{eq:expectationprediction} for the RWM, the nodal domains which Nazarov and Sodin \cite{NS} found were those which were contained in balls of radius $\ll 1/l$ - in some sense the opposite type. We propose the following question:
\begin{quote}
 Is there some $C > 0$ and a sequence of Hecke eigenfunctions $\vp \in \H_l^{\O^\times}$, $l \to \infty$ so that each $\vp$ has a nodal domain contained in a ball of radius $\leq C/l$?
\end{quote}

The proofs of Theorems \ref{totalzeros} and \ref{uniformzeros} boil down to counting the number of real zeros of real polynomials. Returning for a moment to randomness heuristics, the problem of estimating the number of zeros of random polynomials is a classical problem and has been studied by many authors. Here we only have room to mention the classic paper of Kac \cite{KAC} and the more recent paper of Edelman and Kostlan \cite{EK}.

Now we sketch the proof of Theorems \ref{totalzeros} and \ref{uniformzeros}, beginning with Theorem \ref{totalzeros}. At the centre of the proof of Theorem \ref{totalzeros} is the pre-trace formula for the sphere applied to a basis of Hecke eigenfunctions. For any linear combination of Hecke operators specified by a vector $\a$ one obtains a twist of this formula. On one side one has sums of the form
\begin{equation}\label{eq:LHStwist}
\sum_{\phi \in \B_l} A_\a(\phi)^2 \phi(x)\phi(y)
\end{equation}
where $x, y$ are in $S^2$ with $d(x,y) \asymp 1/l$ and the $A_\a(\phi)$ are coefficients depending on the vector $\a$ and $\phi$. On the other side one obtains a bilinear form in $\a$ where the coefficients are given by summing a kernel over elements of $\O$. The main term of this bilinear form is negative definite. Supposing there are few $\phi$ changing sign between $x$ and $y$, so that few of the $\phi(x)\phi(y)$ are negative, we aim to
\begin{enumerate}
\item Choose $\a$ to kill off all the negative contributions to \eqref{eq:LHStwist}.
\item Make the error term of the bilinear form small at $\a$ so that the value of \eqref{eq:LHStwist} remains negative.
\end{enumerate}
This leads to a contradiction unless there are enough sign changes. 

Success depends on the ability to control the `error' bilinear form which is done by bounding the coefficients. This requires a lattice point count due to VanderKam \cite[Lemma 2.1]{VAND} and some dynamical observations appearing in Section \ref{diophantine}. The first of these observations (Lemma \ref{move}) is that if a point $x \in S^2$ is such that $\sqrt{h}x$ has integer coordinates for some small integer $h$, then $\g \in \O$ having small norm either fix $x$ or significantly, depending on $h$ and $n(\g)$, move $x$. This is similar in nature to an observation made by Bourgain and Lindenstrauss \cite[Lemma 3.3]{BL}. The underlying argument is simply that if the distance between two integers is less than 1, they are the same. In Lemma \ref{stabilizer} we observe that if $\sqrt{h} x$ is integral as before then the size of the stabilizer of $x$ in $\O(pq)$ is uniformly bounded through $x$ and primes $p, q$. This depends on a uniform bound for the size of the unit group of imaginary quadratic number fields. To get good bounds for the error term it remains to restrict the support of $\a$. Finally we choose $\a$ so as to kill off terms in \eqref{eq:LHStwist}.

In other settings, apt choice of $\a$ has been called the $\textit{amplification}$ method (when $A_\a$ is to be made big versus the other side of the twisted pre-trace formula) or $\textit{resonator}$ method (when $A_\a$ is to be made small). Amplification usually takes place while establishing subconvexity bounds, appearing for example in \cite{IS},  \cite{VAND} and \cite{MV} to name but a few applications which are relevant to this paper. The resonator method is employed to give $\Omega$ type results for the values in a family as in \cite{SOUND} and \cite{MIL}. Our method could reasonably be called the \textit{annihilator} method.

Now we turn to the proof of Theorem \ref{uniformzeros}. This is entirely analogous to the proof of Theorem \ref{GRStheorem} appearing in \cite{GRS}, although some of the analogies bear explaining. One of the key ingredients is a $L^2$ lower bound for the restriction of a Hecke eigenfunction to a special geodesic. This appears as Lemma \ref{restriction} here, and is much simpler to obtain than the corresponding statement \cite[Theorem 6.1]{GRS} on the modular curve. The subconvex $L^\infty$ bound for the eigenfunctions also plays: we use VanderKam's result (Theorem \ref{vkinfinity}) while Ghosh, Reznikov and Sarnak use the $L^\infty$ bound from \cite{IS}. 

The most interesting modification we make is as follows. In the setting of \cite{GRS} it is important to control the Fourier coefficients of the restriction of a Hecke eigenform $f$ to a geodesic connecting a pair of cusps. It is straightforward that these Fourier coefficients are roughly given by the values $L(1/2 + it , f)$ of the $L$-function associated to $f$ on the critical line. The analogous formula is less straightforward in our setting and can be deduced from work of Waldspurger \cite{WALD}, Jacquet and Chen \cite{JC} and Martin and Whitehouse \cite{MW}. The formula we use is given in Lemma \ref{fourier1}. To get there we adelize the sphere in Section \ref{adelization} and give an account of the period formula in Section \ref{periodsection}. In short, the Fourier coefficients of Hecke eigenfunctions on the sphere around the CM-point $\mathbf{i}$ are roughly given by the central values of $\mathrm{GL}_2 \times \mathrm{GL}_1$ $L$-functions for automorphic representations over $\Q(\sqrt{-1})$. The analogy is more apparent when one writes for $f$ a Maass Hecke eigenform and $\pi_f$ the associated automorphic representation
\begin{equation}
L( 1/2 + it , f) = L( 1/2 , \pi_f \otimes | \bullet |^{it} ) .
\end{equation}
The reader is encouraged to see the letter from Sarnak to Reznikov \cite[Appendix 2]{SARNAKLETTER} for more on the subject of restriction problems and $L$-functions. The rest of the proof of Theorem \ref{uniformzeros} is a straightforward adaptation of methods from the paper \cite{GRS}.

We also mention that in the setting of Theorem \ref{GRStheorem}, Jung \cite{JUNG} has been able get a slightly better exponent (1/16 versus 1/24) without any Lindel\"{o}f hypothesis, but at the large expense that his results only hold true for `almost all' forms in a sufficiently large range. In this regard his result is in the same spirit as Corollary \ref{zeroscoro}. However the exponent obtained by Jung is not as large as in Corollary \ref{zeroscoro} (1/8 versus 1/4)  and his range of spectral parameter is wider. Jung's method is essentially to apply estimates for averaged quantities to the argument of Ghosh, Reznikov and Sarnak from \cite{GRS} whereas ours, as previously outlined, relies on Diophantine analysis and sieving. It has also come to our attention since this paper was written that essentially the same argument concerning Euler's formula in Section \ref{euler} has appeared in work of Jung and Zelditch \cite{JUNGZELDITCH}. We have retained the argument here for completeness.

%%%%%%%%%%%%%%%%%%%%%%%%%%%%%%%%%%%%%%%%%%%%%%%%%%%%%%%%%%%%%%%%%%%%%%%%%%%%%%%%%%%%%%%%%%%%%%%%%%%
%%%%%%%%%%%%%%%%%%%%%%%%%%%%%%%%%%%%%%%%%%%%% TOPOLOGY %%%%%%%%%%%%%%%%%%%%%%%%%%%%%%%%%%%%%%%%%%%%
%%%%%%%%%%%%%%%%%%%%%%%%%%%%%%%%%%%%%%%%%%%%%%%%%%%%%%%%%%%%%%%%%%%%%%%%%%%%%%%%%%%%%%%%%%%%%%%%%%% 

\section{Topology of $\O^\times \backslash S^2$ and Euler's formula}\label{euler}

In this section we give an argument which reduces the counting of nodal domains to counting zeros on certain geodesic arcs.

The fixed points of elements of $\O^\times$ are given by plus or minus their imaginary parts (rescaled so as to have norm 1). A fundamental domain for the action of $\O^\times$ is given by the union of the spherical triangles $T_1$ and $T_2$ as illustrated in Figure \ref{fig:fund}. The triangle $T_1$ has vertices $\mathbf{i}$, $\mathbf{k}$ and $(\mathbf{i} + \mathbf{j} +\mathbf{k}) /\sqrt{3}$ and $T_2$ has vertices $\mathbf{i}$, $\mathbf{k}$ and $(\mathbf{i} - \mathbf{j} + \mathbf{k}) /\sqrt{3}$.

\begin{center} 
\begin{figure}
 \begin{tikzpicture} % CURRENT

\def\R{4} % sphere radius
\def\angEl{30} % elevation angle
\def\angAz{-105} % azimuth angle

\def\angPhi{0} % longitude of point P
\def\angBeta{0} % latitude of point P 'up direction'

\pgfmathsetmacro\H{\R*cos(\angEl)} % distance to north pole
\filldraw[ball color=white] (0,0) circle (\R);
\DrawLatitudeCircle[\R]{0} 

\coordinate[mark coordinate] (N) at (0,\H);
%\coordinate[mark coordinate] (O) at (0,0);
%\draw[dashed] (O) -- (N) +(0.3ex,0.6ex) node[above left] {$\mathbf{k}$};

\LongitudePlane[jkplane]{\angEl}{-25}
\LongitudePlane[ikplane]{\angEl}{-115}

% up angle then scale.
\path[jkplane] (0:\R) coordinate[mark coordinate] (j);
\path[jkplane] (90:\R) coordinate[mark coordinate] (k);
\path[ikplane] (0:\R) coordinate[mark coordinate] (i);
\path[ikplane] (30:\R) coordinate  (reference);

\LongitudePlane[kijk]{\angEl}{-70}
\LongitudePlane[kmijk]{\angEl}{-160}

%\DrawLatitudeCircle[\R]{30} %guideline

\DrawLongitudeCircle[\R]{-25} %new correct Longitudes
\DrawLongitudeCircle[\R]{65}
%\DrawLongitudeCircle[\R]{110} %guideline

% 'diagonal' planes
\LongitudePlane[iijk]{\angEl + 49 }{35}
\LongitudePlane[jijk]{\angEl + 66.5}{-15.5}

\LongitudePlane[imijk]{\angEl + 11}{-35}

\path[kmijk] (30:\R) coordinate[mark coordinate] (mijk);
\path[iijk] (-69:\R) coordinate[mark coordinate] (ijk);

%labels
\draw (j) node[below right] {$\mathbf{j}$};
\draw (i)  +(0.5ex,-1ex) node[below right] {$\mathbf{i}$};
\draw (k) +(0ex,1ex)  node[above] {$\mathbf{k}$}; %can shift label like this
\draw (ijk) +(-1ex,0ex) node[above left] {$\frac{\mathbf{i} + \mathbf{j} + \mathbf{k}}{\sqrt{3}}$};

\draw (mijk) +(+2ex,-2.5ex) node[above right] {$\frac{\mathbf{i} - \mathbf{j} + \mathbf{k}}{\sqrt{3}}$};

\draw (reference) +(4ex,0ex) node[right] {$T_1$};
\draw (reference) +(-2ex,+0.2ex) node[left] {$T_2$};
%set scales
\tikzset{iijk/.prefix style={scale=\R}}
\tikzset{jijk/.prefix style={scale=\R}}
\tikzset{kijk/.prefix style={scale=\R}}

\tikzset{kmijk/.prefix style={scale=\R}}
\tikzset{imijk/.prefix style={scale=\R}}
%              %start      $length$
%\draw[iijknew] (0:1) arc (0:78:1);
\draw[kijk] (0,1) arc(90:30:1); %k to ijk arc
\draw[iijk] (-69:1) arc (-69:-121:1);
\draw[jijk]  (-20:1) arc (-20:-72:1);
\draw[imijk] (-121:1) arc (-121:-173:1);
\draw[kmijk] (0,1) arc(90:30:1);
\end{tikzpicture}
\caption{The fundamental domain for $\O^\times$ is $T_1 \cup T_2$ (up to a measure zero set).}
\label{fig:fund}
\end{figure}
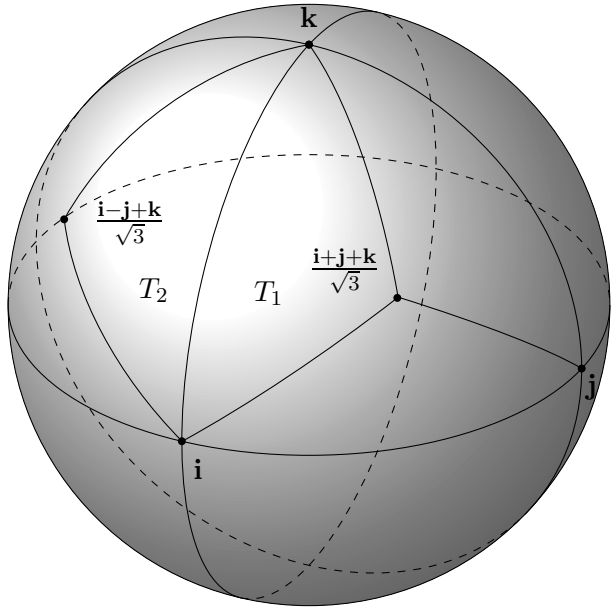
\end{center}
 
In the case that $f \in \H_l^{\O^\times}$ with $l$ \textit{even}, which will be our focus, the function $f$ is actually invariant under a larger group of symmetries which we will call $\G$. It turns out that $\G$ is generated by the projectivization of $\O^\times$ and reflection in the origin. The group $\G$ contains $\O^\times / Z(\O^\times)$ as an index 2 subgroup, so $\G$ has size 24. For us, the most important feature of $\G$ is that it contains the elements
\begin{align*}
R_x : (x , y , z) &\mapsto (-x , y ,z), \\
R_y : (x , y , z) &\mapsto (x , -y ,z) ,\\
R_z : (x , y , z) &\mapsto (x , y , -z) .
\end{align*}
In particular, the $\G$-invariance of $f$ implies that one can view $f$ as a function on $T_1 \cup T_2$ which has matching values on glued boundary pieces \textit{and} which is invariant under the reflection mapping $T_1$ to $T_2$. Now it should be clear that the difference between $l$ even and odd here is like the difference between the even and odd Maass forms in \cite{GRS}.
 We will crucially use the invariant arc of this reflection in what follows, however firstly we must give the necessary properties of nodal sets.

Recall that the nodal set of $f$ is the set
\begin{equation}
Z(f) = \{ v \in S^2 \: : \: f(v) = 0 \} = f^{-1}(0).
\end{equation}
A nodal domain is a connected component of $S^2 - Z(f)$. We write $\N(f)$ for the number of nodal domains of $f$.

The proofs of the following results appear in the paper of Cheng \cite{CHENG}.
\begin{thm}[Cheng] 
For any eigenfunction $f$ of $\Delta_{S^2}$
\begin{enumerate}
\item The nodal set of $f$ is a union of $C^2$-immersed circles.
\item When branches of the nodal set of $f$ cross, the tangent vectors at the crossing point form an equiangular system. In particular, finitely many branches cross at a given point.
\end{enumerate}
\end{thm}
Now we prove that, given invariance of a spherical harmonic $f$ under a reflection $R$, finding zeros of $f$ on the fixed equator of $R$ provides roughly as many nodal domains of $f$. Our argument is essentially the same as that of Ghosh, Reznikov and Sarnak in the proof of \cite[Theorem 2.1]{GRS}.
\begin{lemma} \label{lemma:euler} Suppose that $l$ is even and $f \in H^{\O^\times}_l$ has $N$ zeros on one of the geodesics $\{ x = 0 \}$, $\{ y = 0 \}$ or $\{ z = 0 \}$, and does not vanish identically on the geodesic. Then 
\begin{equation}
\N(f) \geq 1 + N/2.
\end{equation}
\end{lemma}
\begin{proof}
We consider the hemispheres $H_+ = \{ x \geq 0 \}$ and $H_- = \{ x \leq 0 \}$. Assume $f$ has $N$ zeros on $\d H_+ = \{ x = 0 \}$ ($f$ actually has the same number of zeros on each of the geodesics mentioned in the Lemma). Let $V$ consist of all the zeros of $f$ on $\d H_+$. Now add to $V$ all the points in $Z(f) \cap H_+$ which are crossing points (i.e. the nodal set does not look like $(0,1)$ in any neighbourhood of that point). We define a multigraph $G$ which has vertex set $V$ and an edge set $E$ with an edge between $v_1$ and $v_2$ each time
\begin{enumerate}
\item There is a nonbacktracking path in $Z(f) \cap H_+$ between $v_1$ and $v_2$ which does not pass through any other vertex. In this case we identify the edge with the underlying set of the path (which is homeomorphic to an interval).
\item The vertices $v_1$ and $v_2$ both lie on $\d H_+$ and are adjacent on $\d H_+$. In this case we identify the edge with the geodesic arc on $\d H_+$ along which $v_1$ and $v_2$ are adjacent.
\end{enumerate}
Thus $G$ comes with a planar embedding into $H^+$. We will write $F$ for the faces associated to this embedding. For a face $S \in F$  either $S$ does not meet $\d H_+$ along any edge, in which case $S \cup R_x S$ consists of two disjoint nodal domains of $f$, or $S$ meets $\d H_+$ along an edge, in which case $S \cup R_x S$ yields one nodal domain of $f$. As $f$ is $R_x$-invariant it is clear that distinct faces cannot lie in the same nodal domain. Therefore
\begin{equation}\label{eq:faceineq}
\N(f) \geq |F|.
\end{equation}
We will count vertices and edges and use Euler's formula $|V| - |E| + |F| = 1$ (throwing out the face exterior to $H_+$ viewed as a disc in $\R^2$). Every vertex in the multigraph $G$ has least $3$ incident edges as follows. If the vertex lies on $\d H_+$ then it has two incident arcs on $\d H_+$ and at least one other incident edge going into $H_+$ due to the invariance of the nodal set under the reflection $R_x$ and that $f$ is not identically zero on $\delta H_+$. If the vertex is in the interior of $H_+$ then it arose from a crossing point in which case it has more than $2$ incident edges. Therefore
\begin{equation}
|E| \geq \frac{3}{2} |V| 
\end{equation}
so by Euler's formula
\begin{equation}
|F| \geq 1 + |V|/2 \geq 1 + N/2
\end{equation}
as $|V| \geq N$. Together with \eqref{eq:faceineq} this gives the desired
\begin{equation}
\N(f) \geq 1 + N/2.
\end{equation}

\end{proof}
 
%%%%%%%%%%%%%%%%%%%%%%%%%%%%%%%%%%%%%%%%%%%%%%%%%%%%%%%%%%%%%%%%%%%%%%%%%%%%%%%%%%%%%%%%%%%%%%%%%%%
%%%%%%%%%%%%%%%%%%%%%%%%%%%%%%%%%%%%%%%%%%%%% DIOPHANTINE ESTIMATES %%%%%%%%%%%%%%%%%%%%%%%%%%%%%%%
%%%%%%%%%%%%%%%%%%%%%%%%%%%%%%%%%%%%%%%%%%%%%%%%%%%%%%%%%%%%%%%%%%%%%%%%%%%%%%%%%%%%%%%%%%%%%%%%%%%
\section{Diophantine estimates}\label{diophantine}
In this section we present Diophantine estimates which will be key in finding sign changes amongst the values of the Hecke eigenfunctions on $S^2$. Recall that the spherical distance between points $x , y \in S^2$ is given by
\begin{equation}
d(x,y) = \cos^{-1}( \langle x,y \rangle_{\R^3} ). 
\end{equation}

The first Lemma is a direct result of Lemma 2.1 in the paper of VanderKam \cite{VAND}.
\begin{lemma}[VanderKam]\label{counting} For $x \in S^2$ let $N( x, m, \delta)$ denote the number of elements $\g \in \O$ whose norm is $m$ and $d(x, \g. x) \leq \delta$ or $d(x , \g.x) \geq \pi - \delta$. Then
\begin{equation}
N( x, m, \delta) \ll_\e  \left\{ \begin{array}{rl}
      m^\e(\delta^{1/2}m + 1) & \text{if } \delta < 1/m  \\
      m^{\e}(m^{1/2} + \delta^{2/3}m) & \text{else } .
    \end{array}\right.
\end{equation}
\end{lemma} 
\begin{rmk} Lemma 2.1 in \cite{VAND} applies to any point $x \in S^2$. We will apply Lemma \ref{counting} to points $(a , b , 0) / \sqrt{a^2 + b^2}$ in which case the proof can be made slightly simpler than the one which VanderKam gives.
\end{rmk}  
As remarked the preceding estimate does not care about any algebraic properties of $x \in S^2$. To proceed we define the $\textit{height}$ $h(x)$ of $x \in S^2$ to be the least positive integer $h$ such that $\sqrt{h}x$ has integer coordinates (if it exists). The next Lemma says that if a point has small height, the elements of $\O$ with small norm and which do not fix the point or map it to its antipode move the point (and not close to its antipode).

\begin{lemma}\label{move}
Let $x$ in $S^2$ have height $h(x) \in \Z$. Then for $\g \in \O$, if $\g . x \notin \{ x , -x \}$ then
\begin{equation}
\pi - \frac{1}{\sqrt{n(\g) h(x)}} \geq d( \g . x ,x ) \geq \frac{1}{\sqrt{n(\g) h(x)}} .
\end{equation}
\end{lemma}
\begin{proof}
Note that $\sqrt{h(x)} x$ has integer coefficients in $\R^3$, therefore for $\g \in \O(m)$ an easy calculation gives that $\g \sqrt{h(x)} x \bar{\g}$ has half integer coefficients in $\R^3$. If $\g . x \notin \{ x , -x \}$ then 
\begin{equation}
-n(\g)h(x) < \langle \g \sqrt{h(x)} x \bar{\g} , \sqrt{h(x)} x \rangle = n(\g) h(x) \langle \g . x ,x \rangle < n(\g) h(x)
\end{equation}
and the expression in the middle is a half integer. This gives 
\begin{equation*}
-n(\g)h(x) + 1/2 \leq n(\g) h(x) \langle \g . x ,x \rangle \leq n(\g) h(x) - 1/2 
\end{equation*} 
which yields the needed inequalities after noting 
\begin{equation}
1 - d(\g . x , x )^2 /2 \leq \langle \g . x ,x \rangle \leq 1 - \frac{1}{2 n(\g)h(x)},
\end{equation}
\begin{equation}
 -1 + (\pi - d(\g . x , x))^2 /2 \geq      \langle \g . x , x \rangle  \geq -1 +  \frac{1}{2 n(\g)h(x)}.
\end{equation}
\end{proof}

It remains in this section to control the number of $\g \in \O$ of given norm and with $\g . x \in \{ x, - x \}$, with a uniform bound through all $x \in S^2$ having any height and the norm of $\g$ having a bounded number of prime divisors - for our purposes 2 prime divisors suffices.

\begin{lemma}\label{stabilizer}
There is an absolute constant $N_0$ such that for any $x \in S^2$ with existent height $h(x)$ and $p,q$ (not necessarily distinct) primes the number of $\g \in \O$ with $n(\g) = pq$ and $\g . x \in \{ x , -x \} $ is bounded by $N_0$. In other words for $x$ having any height $h(x) < \infty$
\begin{equation}
N(x, pq, 0 ) \ll 1.
\end{equation}
\end{lemma}

\begin{proof}
With everything as in  Lemma \ref{stabilizer}, the point $\sqrt{h(x)}x$ is contained in the image of an embedding of normed $\Q$-algebras $\iota : F = \Q(\sqrt{-h(x)}) \hookrightarrow B(\Q)$, with $\iota(\sqrt{-h(x)}) = \sqrt{h(x)}x$. The rotations in $B(\Q)$ fixing $x$ are precisely those in $\iota(F)$. Let $\O_F$ denote the maximal order in $F$.

The result will follow quickly from the fact that for any quadratic imaginary number field there is an absolute constant $\Omega$ such that the number of elements of $\O_F$ of norm either $4 p q$ or $4p^2 q^2$ is bounded by $\Omega$. Consider the case of counting elements of norm $4p^2q^2$. Any element $x \in \O_F$ of norm $4p^2q^2$ generates an ideal $(x)$ of the same norm. We have then uniquely
\begin{equation}
(x)  = \prod_{i} \P_i
\end{equation}
with each $\P_i$ lying over $2$, $p$ or $q$ and at most 2 of the $\P_i$ lying over each of  $2$, $p$, $q$. There are therefore at most $3^3 = 27$ choices (if $p = q$ there are even less choices) for the principal ideal $(x)$ and each principal ideal corresponds to at most $|\O_F^\times |$ elements. As $| \O_F^{\times} |$ has a universal bound for quadratic imaginary number fields $F$ these arguments give a universal bound for the number of elements of $\O_F$ of norm $4p^2q^2$. The case of norm $4pq$ is similar.

If $\g \in \O(pq)$ and $\g  . x  = x$ then $2 \g \in \iota(\O_F)$ with norm $4 p q$, so there are at most $\Omega$ possibilities for such $\g$.

Suppose there is some $\g_0 \in \O(pq)$ such that $\g_0 . x = - x$ (or else we are done). Then for any $\g_1 \in \O(pq)$ with $\g_1 . x = -x$ we have $\g_1 . \g_0 . x = x$ and hence $2 \g_1 \g_0 \in \iota(\O_F)$ with norm $4 p^2 q^2$. Moreover as $\g_1$ ranges while $\g_0$ is fixed the elements of $\O_F$ produced are distinct. Hence there are at most $\Omega$ possibilities for $\g \in \O(pq)$ with $\g x = -x$. This concludes the proof.
\end{proof}

%%%%%%%%%%%%%%%%%%%%%%%%%%%%%%%%%%%%%%%%%%%%%%%%%%%%%%%%%%%%%%%%%%%%%%%%%%%%%%%%%%%%%%%%%%%%%%%%%%%
%%%%%%%%%%%%%%%%%%%%%%%%%%%%%%%%%%%%%%%%%%%%% KERNEL ESTIMATES %%%%%%%%%%%%%%%%%%%%%%%%%%%%%%%%%%%%
%%%%%%%%%%%%%%%%%%%%%%%%%%%%%%%%%%%%%%%%%%%%%%%%%%%%%%%%%%%%%%%%%%%%%%%%%%%%%%%%%%%%%%%%%%%%%%%%%%%
\section{Kernel estimates}\label{summationestimates}
Recall the decomposition 
\begin{equation}
L^2(S^2) = \overline{\bigoplus}_l \H_l
\end{equation}
and denote by $P_l$ the projection onto the space $\H_l$. It is well known that $P_l$ is a integral operator with kernel given by the formula
\begin{equation}\label{eq:kernel1}
[P_l \phi ] (x) = \frac{2l+1}{4\pi} \int_{S^2} p_l(\langle x,y \rangle ) \phi(y) dS^2(y) , \quad \phi \in L^2(S^2)
\end{equation}
where $p_l$ is the $l$th Legendre polynomial. We will look for estimates for sums
\begin{equation}
\sum_{\g \in \O(m)} p_l( \langle \g . x , y \rangle )
\end{equation}
for $x$ close to $y$ in anticipation of the upcoming annihilator method. We need some facts about $p_l$ to get started. Firstly
\begin{equation}
p_l(1) = (-1)^l p_l(-1) = 1
\end{equation}
gives the value of $p_l( \langle x,y \rangle )$ when $y = x , -x$. 

\textit{In what follows, we will restrict to $l$ even}. 

In particular for us $p_l(1) = p_l(-1) = 1$. In general $p_l$ satisfies the bound
\begin{equation}\label{eq:plsize}
p_l(\cos(\theta)) \ll \min \left( 1 , (l\sin \theta)^{-1/2} \right)
\end{equation}
which follows from \cite[p. 165 Theorem 7.3.3]{SZEGO} together with $p_l(\cos(\theta)) \leq 1$ \cite[ p. 164 Theorem 7.3.1]{SZEGO}. It can be deduced from `Hilb's asymptotics' \cite[ p. 195 Theorem 8.21.6]{SZEGO} that  
\begin{equation}\label{eq:value}
p_l \left( \cos \left( \frac{3\pi}{2l + 1} \right) \right) \asymp J_0 \left( \frac{3\pi}{2} \right) \approx -0.26
\end{equation}
as in the B.A. Thesis of Nastasescu \cite{NAST} (following Nazarov and Sodin \cite{NS}). Here $J_0$ denotes the zeroth Bessel function. As the value will be repeatedly used we introduce notation
\begin{equation}
W_l = \frac{ 3\pi }{2l+1} ,
\end{equation}
 which can be thought of as the first wavelength of the spherical eigenfunction of energy $l(l+1)$.  
 
Some of the estimates in this section depend on `integration by parts' arguments. It will therefore be convenient for purely technical reasons to prepare the next Lemma which follows straightforwardly from Lemma \ref{counting} and applications of the spherical triangle inequality; the proof is omitted.
\begin{lemma} \label{book}
When $l \gg m$ with sufficiently large implied constant, $x , y \in S^2$ and $d(x,y) \leq W_l = \frac{3\pi}{2l+1}$ the quantity
\begin{equation}
N( x, y, m, \delta) \equiv | \{ \g \in \O : \: n(\g) = m , \: d(\g.x , y) \leq \delta \text{ or } d(\g.x , y) \geq \pi - \delta  \: \}|
\end{equation}
satisfies
\begin{equation}
N( x, y, m, \delta) \ll_\e  \left\{ \begin{array}{rl}
      m^{\e} & \text{if } \delta < 1/l \\
      m^\e(\delta^{1/2}m + 1) & \text{if } 1/ l \leq \delta < 1/m - W_l \\
      m^{\e}(m^{1/2} + \delta^{2/3} m)   & \text{if } 1/m - W_l \leq \delta  .
    \end{array}\right.
\end{equation}
\end{lemma}

We can now give estimates for sums of the kernel $p_l$ over Hecke elements.

\begin{lemma}\label{sumestimate}
There exist $C_1$ and $C_2$ absolute such that for even $l$ and $x,y \in S^2$, if the following hold:
\begin{enumerate}
\item $d(x,y) \leq W_l = \frac{3\pi}{2l+1}$,
\item $l \geq C_1 m$,
\item The height of $x$, $h(x)$ exists and $l \geq C_2 \sqrt{ m h(x) }$,
\end{enumerate}
then
\begin{equation}\label{eq:sum}
\sum_{\g \in \O(m) } p_l (\langle \g . x , y \rangle ) = p_l ( \langle x , y \rangle ) N(x,m,0)   + O_\e( m^\e l^{-1/2}( mh^\e + m^{1/4}h(x)^{1/4} ) ).
\end{equation}
\end{lemma}
 
\begin{proof}
With everything as above, Lemma \ref{move} implies that if $\g \in \O(m)$ and $\g . x \notin \{ x , - x \}$ then
\begin{equation}
  \frac{1}{\sqrt{m h(x)}} \leq d( \g . x ,x ) \leq  \pi - \frac{1}{\sqrt{m h(x)}}
\end{equation} so that by the triangle inequality
\begin{equation}
 \frac{1}{\sqrt{m h(x)}} - W_l \leq d(\g  . x , y ) \leq \pi - \frac{1}{\sqrt{m h(x)}} + W_l .
\end{equation}
If $l \geq C_2 \sqrt{ m h(x) }$ with $C_2$ large enough then 
\begin{equation}\label{eq:beginslate}   
\frac{1}{\sqrt{m h(x)}} - W_l > \frac{1}{l} ,
\end{equation}
we pick such a $C_2$. Consequently, aside from $\g$ with $\g . x \in  \{ x , - x \}$, summation in \eqref{eq:sum} begins late enough to get a good bound. The previous discussion gives
\begin{align}\label{eq:est1}
\sum_{\g \in \O(m) } p_l (\langle \g . x , y \rangle ) &= \sum_{\substack {\g \in \O(m) \\ \g.x \in \{ x , -x \}}}   p_l (\langle \g . x , y \rangle)  +\\
& \sum_{\substack{ \g \in \O(m) \\  \frac{1}{\sqrt{m h(x)}} - W_l \leq d(\g . x , y )  \leq \pi - \frac{1}{\sqrt{m h(x)}} + W_l}}   p_l (\langle \g . x , y \rangle ) . \nonumber
\end{align}
The first term on the right hand side can be written
\begin{equation}\label{eq:est2}
\sum_{\substack {\g \in \O(m) \\ \g.x \in \{ x , -x \}}}   p_l (\langle \g . x , y \rangle)  = p_l (\langle  x , y \rangle) N(x,m,0)
\end{equation}
because $l$ is even. The remainder is estimated in size by the triangle inequality and the known bound \eqref{eq:plsize} for $p_l$:
\begin{equation}\label{eq:est3}
\left| \sum_{\substack{ \g \in \O(m) \\ \frac{1}{\sqrt{m h(x)}} - W_l \leq d(\g . x , y )  \leq \pi - \frac{1}{\sqrt{m h(x)}} + W_l}}   p_l (\langle \g . x , y \rangle )\right| \ll  l^{-1/2} \sum_{\delta \geq   \frac{1}{\sqrt{m h(x)}} - W_l } \delta^{-1/2} dN(x,y,m,\delta) ,
\end{equation}
where 
\begin{equation}
dN(x,y,m,\delta) \equiv  | \{ \g \in \O : \: n(\g) = m , \: d(\g.x , y) = \delta \text{ or } d(\g.x , y) = \pi - \delta  \: \}| 
\end{equation}
and the summation of $\delta$ values is over a finite set containing all attained values of $d(\g.x , y )$ or $\pi-d(\g.x,y)$ for $\g \in \O(m)$.
We split up summation into the ranges appearing in Lemma \ref{book}, and specify $C_1$ so that the results from Lemma \ref{book} apply. Then
\begin{align}\label{eq:est4}
 \sum_{\delta \geq   \frac{1}{\sqrt{m h(x)}} - W_l} \delta^{-1/2} dN(x,y,m,\delta) &= \sum_{  \frac{1}{\sqrt{m h(x)}} - W_l \leq \delta < \frac{1}{m} - W_l }  \delta^{-1/2} dN(x,y,m,\delta) \\
&+ \sum_{  \frac{1}{m} - W_l \leq \delta  }  \delta^{-1/2} dN(x,y,m,\delta)  \nonumber \\
& \ll_\e m^\e ( mh^\e  + m^{1/4}h(x)^{1/4} ) \nonumber
\end{align}
by an `integration by parts' argument making use of \eqref{eq:beginslate} and the estimates of Lemma \ref{book}. We remark that if $h(x) \leq m$ then the range  $\frac{1}{\sqrt{m h(x)}} - W_l \leq \delta < \frac{1}{m} - W_l$ is empty and the above calculation is still valid, but simpler. Putting our estimates \eqref{eq:est1}, \eqref{eq:est2}, \eqref{eq:est3} and \eqref{eq:est4} together we get the desired 
\begin{equation}
\sum_{\g \in \O(m) } p_l (\langle \g . x , y \rangle ) =  p_l (\langle  x , y \rangle) N(x,m,0) + O_\e( m^\e l^{-1/2} ( mh^\e  + m^{1/4}h(x)^{1/4} )  ) .
\end{equation}
\end{proof}  
  
%%%%%%%%%%%%%%%%%%%%%%%%%%%%%%%%%%%%%%%%%%%%%%%%%%%%%%%%%%%%%%%%%%%%%%%%%%%%%%%%%%%%%%%%%%%%%%%%%%%
%%%%%%%%%%%%%%%%%%%%%%%%%%%%%%%%%%%%%%%%%%%%% ANNIHILATOR %%%%%%%%%%%%%%%%%%%%%%%%%%%%%%%%%%%%%%%%%%%
%%%%%%%%%%%%%%%%%%%%%%%%%%%%%%%%%%%%%%%%%%%%%%%%%%%%%%%%%%%%%%%%%%%%%%%%%%%%%%%%%%%%%%%%%%%%%%%%%%%
\section{The spectral annihilator}
The kernel of the orthogonal projection $P_l : L^2(S^2) \to \H_l$ can be written in two ways. Firstly we extend the orthonormal basis $\B_l$ of Hecke eigenfunctions in $\H_l^{\O^\times}$ to an orthonormal basis $\hat{\B_l}$ of $\H_l$. Then the kernel of $P_l$ is implicitly given by the formula 
\begin{equation}\label{eq:kernel2}
[P_l \varphi ] (x) =  \int_{S^2} \left( \sum_{\phi \in \hat{\B_l}} \phi(x)\phi(y) \right) \varphi(y) dS^2(y) , \quad \varphi \in L^2(S^2) .
\end{equation}
On the other hand we have the expression for the kernel from equation \eqref{eq:kernel1} which when taken together with \eqref{eq:kernel2} gives the following special form of the pre-trace formula on $S^2$:
\begin{equation}\label{eq:pretrace}
\frac{4\pi}{2l + 1} \sum_{\phi \in  \hat{\B_l}} \phi(x)\phi(y) = p_l( \langle x , y \rangle) .
\end{equation} 
The basis $\B_l$ consists of $\phi$ which are eigenfunctions of all the $T_m$ for $m$ odd, with $T_m \phi = \lambda_\phi(m) \phi$. We introduce a new normalized variable $\eta_\phi(m) \equiv \lambda_\phi(m)/ m^{1/2}$. The $\eta_\phi(m)$ are real valued and Deligne's bound \cite{DEL} on the coefficients of modular forms\footnote{This is included to orient the reader, we do not use it.} gives
\begin{equation}
\eta_{\phi}(m) \ll_\e m^{\e} .
\end{equation} 
For any function $\a : \mathbb{N} \to \R$ of sufficient decay (in what follows $\a$ will be of finite support) we can form the operator
\begin{equation*}
\left( \sum_{n \textrm{ odd}}  \frac{ \a(n) T_n }{ \sqrt{n} } \right)^2
\end{equation*}
and apply it to the $x$ variable on both sides of \eqref{eq:pretrace} to obtain
\begin{equation}\label{eq:twistedpretrace}
\frac{4\pi}{2l + 1} \sum_{\phi \in \B_l} \left( \sum_n \a(n) \eta_{\phi}(n) \right)^2 \phi(x)\phi(y) = \sum_{n,m} \frac{\a(n)\a(m)}{\sqrt{nm}} [T_n T_m p_l( \langle \bullet ,y\rangle) ](x) .
\end{equation} 
Note that the sum on the left hand side of \eqref{eq:twistedpretrace} is over $\B_l$ as the $T_m$ annihilate the orthogonal complement of $\H_l^{\O^\times}$ in $\H_l$ when $m$ is odd.
From now on we will write, for $\phi \in \B_l$,
\begin{equation}
A_\a(\phi) \equiv   \sum_n  \a(n)\eta_\phi(n) 
\end{equation}
which when squared gives the coefficient of $\phi(x)\phi(y) $ in \eqref{eq:twistedpretrace}.
Using the Hecke recursion \eqref{eq:recursion} we can write 
\begin{equation}
\frac{4\pi}{2l + 1} \sum_{\phi \in \B_l} A_\a(\phi)^2 \phi(x)\phi(y) = | \O^\times |  \sum_{n,m} \a(n)\a(m) \sum_{d | (n,m)} \frac{d}{ \sqrt{nm} } [T_{nm/d^2} p_l( \langle  \bullet ,y \rangle ) ](x) ,
\end{equation}
and recalling the definition of the Hecke operator $T_{nm/d^2}$ from \eqref{eq:hecke} we obtain our main spectral equation:
\begin{equation}\label{eq:spectral}
\frac{4\pi}{| \O^\times |(2l + 1)} \sum_{\phi \in \B_l} A_\a(\phi)^2 \phi(x)\phi(y) =  
 \sum_{n,m} \a(n)\a(m) \sum_{d | (n,m)} \frac{d}{ \sqrt{nm} } \sum_{\g \in \O(\frac{nm}{d^2})} p_l ( \langle \g . x , y \rangle ). 
\end{equation}

The Hecke summation estimates which we have prepared in Section \ref{summationestimates} are precisely for the purpose of controlling the right hand side of \eqref{eq:spectral}. In what follows we aim to find large subspaces of $l^2$ on which the quadratic form in \eqref{eq:spectral} is negative definite. This will enable us to find enough negative contributions from the left hand side, hence sign changes of eigenfunctions between $x$ and $y$. As a technical aside, in what follows we will take the approach of confining $\e$ considerations to the support of $\a$.
  
\begin{lemma}[Key Lemma for `on average' result]\label{key}%%%%%%%%%%%%%%%%%%%%%%%
There exists an absolute $K > 1$ such that for any $\e > 0$ there exists $L = L(\e)$ with
\begin{equation}
 \sum_{\phi \in \B_l} A_\a(\phi)^2 \phi(x)\phi(y) < 0 
\end{equation}
whenever
\begin{enumerate}
\item  $l$ is even and $l > L(\e)$ 
\item  $x,y \in S^2$ with $d(x,y) = W_l$ and $h(x) \leq l^{3/2}$
\item  $\| \a \|_{l^2} = 1 $ and $\mathrm{supp}(\a) \subseteq [l^{(1/4 -2\e)/K}, l^{1/4-2\e}] \cap \{ \mathrm{primes} \}$.
\end{enumerate}
\end{lemma}

\begin{proof}
Let $C \equiv - J_0(3\pi/2) > 0 $ throughout this proof. We will take $\a : \mathbb{N} \to \R$ as in the Lemma and denote by $S$ the primes in $[l^{(1/4 -2\e)/K}, l^{1/4-2\e}]$ for $K$ to be chosen later. We split the right hand side of the spectral equation \eqref{eq:spectral} into summands with $d = 1$ (which appear for all pairs $p,q \in S$ and are the only contribution when $p \neq q$) and summands where $d = p  = q$. This gives
\begin{align}\label{eq:main}
 \frac{4\pi}{| \O^\times |(2l + 1)}  \sum_{\phi \in \B_l} A_\a(\phi)^2 \phi(x)\phi(y) &= \sum_{p \in S} \a(p)^2 \sum_{\g \in \O(1) = \O^\times} p_l (  \langle \g . x , y \rangle )   \\
&+ \sum_{p , q \in S} \a (p)\a (q)  \frac{1}{ \sqrt{pq} } \sum_{\g \in \O(pq)} p_l ( \langle \g . x , y \rangle ). \nonumber
\end{align}
The first term on the right hand side can be estimated by using Lemma \ref{sumestimate} with $m = 1$ (which is valid given the assumptions of the Lemma) and $h(x) \leq l^{3/2}$ giving
\begin{equation}\label{eq:unitsum}
\sum_{\g \in \O(1) = \O^\times} p_l (  \langle \g . x , y \rangle ) = p_l(\langle  x , y \rangle ) N(x,1,0) + O( l^{-1/8} ).
\end{equation}
As $\| \a \|_{l^2} = 1$ it follows from the above that
\begin{equation}
\sum_{p \in S} \a(p)^2 \sum_{\g \in \O(1) = \O^\times} p_l (  \langle \g . x , y \rangle )  = p_l(\langle  x , y \rangle ) N(x,1,0) + O( l^{-1/8} ) 
\end{equation}
Noting that $N(x,1,0) \geq 2$ (contributions coming from $\pm 1 \in \O^\times$) and using the asymptotic value from \eqref{eq:value} implies that for any $\delta > 0$ there is $L$ large enough so that when $l > L$
\begin{equation}
\sum_{p \in S} \a(p)^2 \sum_{\g \in \O(1) = \O^\times} p_l (  \langle \g . x , y \rangle ) \leq - 2 C + \delta .
\end{equation}
We choose $\d = C/2$ and $L(\e)$ large enough so that
\begin{equation}\label{eq:mainterm}
\sum_{p \in S} \a(p)^2 \sum_{\g \in \O(1) = \O^\times} p_l (  \langle \g . x , y \rangle ) \leq - 3C /2 
\end{equation}
which will be the known negative contribution from \eqref{eq:main}. The remainder will be treated as an error term. Indeed, considering the second term on the right hand side of \eqref{eq:main}, under the assumptions of the Lemma it is easy to check that the requisites of Lemma \ref{sumestimate} are satisfied when $l$ is large enough. Using Lemma \ref{sumestimate} together with the triangle inequality gives the following bound:
\begin{align}\label{eq:calc1}
&\left| \sum_{p , q \in S} \a (p)\a (q)  \frac{1}{ \sqrt{pq} } \sum_{\g \in \O(pq)} p_l ( \langle \g . x , y \rangle ) \right| \leq \\ 
& \sum_{p , q \in S} |\a (p)||\a (q)| \frac{1}{\sqrt{pq}} \left( |p_l( \langle x , y \rangle ) N(x,pq,0)| + O_\e \left( (pq)^\e l^{-1/2}(pq h^\e + (pq)^{1/4}h(x)^{1/4}) \right) \right) \nonumber.
\end{align}
Now we use the estimate $N(x,pq,0) \leq N_0$ from Lemma \ref{stabilizer} together with $|p_l(\langle x , y \rangle )| \leq 1$ to continue, writing $h = h(x)$,
\begin{align*}
&  \sum_{p , q \in S} |\a (p)||\a (q)| \frac{1}{\sqrt{pq}} \left( |p_l( \langle x , y \rangle ) N(x,pq,0)| + O_\e \left( (pq)^\e l^{-1/2}(pq h^\e + (pq)^{1/4}h^{1/4}) \right) \right) \\
&\leq \sum_{p , q \in S} |\a (p)||\a (q)| \left( N_0(pq)^{-1/2} + O_\e \left(  (pq)^{1/2 + \e} h^\e l^{-1/2} + (pq)^{-1/4 + \e}h^{1/4} l^{-1/2} \right)  \right) \\
&= N_0 \left( \sum_{p \in S} |\a(p)| p^{-1/2} \right)^2 + O_\e\left( l^{-1/2} h^\e \left( \sum_{p \in S} |\a(p)| p^{1/2 + \e} \right)^2 \: \right)\\
&+ O_\e \left(l^{-1/2} h^{1/4} \left( \sum_{p \in S} |\a(p)| p^{-1/4 + \e} \right)^2 \:\right) \\
&\leq  N_0 \left( \sum_{p \in S} \frac{1}{p} \right) + O_\e \left( l^{-1/2} h^\e \left( \sum_{p \in S}  p^{1 + 2\e} \right) + l^{-1/2} h^{1/4} \left( \sum_{p \in S}  p^{-1/2 + 2\e} \right) \right)
\end{align*}
where we used Cauchy-Schwarz and $\| \a \|_{l^2} = 1$ for the last inequality. We will now bound the sums in the last line one by one. Mertens' second Theorem gives
\begin{equation}
\sum_{p \in S} \frac{1}{p} = \sum_{l^{(1/4 - 2\e)/K} \leq p \leq l^{1/4 - 2\e}} \frac{1}{p} = \log K + o(1) , 
\end{equation}
and we choose $K = e^{C/(2N_0)} > 1$ so that $N_0 \log K = C / 2$. The other sums are estimated more coarsely by 
\begin{equation}
\sum_{p \in S}  p^{1 + 2\e} \leq \sum_{0 < n \leq  l^{1/4 -2\e} }n^{1 + 2\e} \ll_\e l^{(1/4 - 2\e)(2 + 2\e)} \ll l^{1/2 - 3\e}
\end{equation}
and 
\begin{equation}
\sum_{p \in S}  p^{-1/2 + 2\e} \leq \sum_{0 < n \leq  l^{1/4 -2\e} }n^{-1/2 + 2\e}  \ll_\e l^{(1/4 - 2\e)(1/2 + 2\e)} \leq l^{1/8 - \e/2}.
\end{equation}
Using that $h \leq l^{3/2}$ and incorporating our estimates gives for the error
\begin{equation}
  N_0 \left( \sum_{p \in S} \frac{1}{p} \right) + O_\e \left( l^{-1/2} h^\e\left( \sum_{p \in S}  p^{1 + 2\e} \right) + l^{-1/2} h^{1/4} \left( \sum_{p \in S}  p^{-1/2 + 2\e} \right) \right) = C/2 + o_\e(1)
 \end{equation}
so that in total (tracing the previous series of bounds back to \eqref{eq:calc1} and recalling \eqref{eq:main} and \eqref{eq:mainterm})
\begin{equation}
\frac{4\pi}{| \O^\times |(2l + 1)}  \sum_{\phi \in \B_l} A_\a(\phi)^2 \phi(x)\phi(y) \leq -C + o_\e(1)
\end{equation}
which concludes the proof.
\end{proof}

The next Proposition uses the previous estimates to find sign changes over the wavelength $W_l$ amongst the Hecke eigenfunctions.

\begin{prop}[Sign changes at algebraic points]\label{prop:signchange} As even $l$ varies for fixed $\e$
\begin{equation}
| \{ \phi \in \H_l^{\O^\times} : \phi \text{ a normalized Hecke eigenfunction}, \: \phi(x)\phi(y) < 0 \}| \gg_\e l^{1/4 - \e}
\end{equation}
uniformly through $x \in S^2$ with $h(x) \leq l^{3/2}$ and $y \in S^2$ with $d(x,y) = W_l$.
\end{prop}

\begin{proof} Fix $\e > 0$ and let $K, L(\e)$ be the constants from Lemma \ref{key} so that when $l > L(\e)$ the quadratic form in $\a$ 
\begin{equation}\label{eq:negative}
 \sum_{\phi \in \B_l}A_\a(\phi)^2 \phi(x)\phi(y)
\end{equation}
is negative when $x,y$ are as in the Proposition and $\a$ is supported on primes in $[l^{(1/4 -2\e)/K}, l^{1/4-2\e}]$. The dimension of the space of such $\a$ is $\gg_\d l^{1/4 - 2\e - \delta}$ after a calculation with the prime number theorem. In particular there is $c = c(\e)$ and $L = L(\e)$ such that when $l > L(\e)$
\begin{equation}
\dim \{ \a : \mathrm{supp}(\a) \subset [l^{(1/4 -2\e)/K}, l^{1/4-2\e}] \cap \{ \text{primes} \} \} \geq cl^{1/4 - 3\e}.
\end{equation}
If there are fewer than $cl^{1/4 - 3\e}$ of the $\phi \in \B_l$ such that $\phi_{l,i}(x) \phi_{l,i}(y) < 0$ then we can find $\a$ with support in $[l^{(1/4 -2\e)/K}, l^{1/4-2\e}]$ and such that
\begin{equation}
A_\a(\phi) = \sum_{n} \a(n) \eta_{\phi}(n) = 0 
\end{equation}
whenever $\phi(x) \phi(y) < 0$; this is just the observation that in a vector space with inner product of dimension $N$ one can find a vector orthogonal to any $N-1$ vectors. For this orthogonal $\a$ all the summands in \eqref{eq:negative} are nonnegative giving the required contradiction, so that at least $cl^{1/4 - 3\e}$ of the $\phi \in \B_l$ sign change between $x$ and $y$.
\end{proof}
\begin{rmk} Theorem \ref{vkinfinity}, which is the $L^\infty$ bound of VanderKam
\begin{equation}\label{eq:linfty}
\| \phi \|_{\infty} \ll_\e l^{5/12 + \e}, \quad \phi \in \B_l
\end{equation}
implies a weaker version of Proposition \ref{prop:signchange} with the lower bound $l^{1/4- \e}$ replaced by $l^{1/6-\e}$. Indeed, in the setting of Proposition \ref{prop:signchange} the pre-trace formula \eqref{eq:pretrace} together with the previously calculated \eqref{eq:unitsum} gives 
\begin{equation}\label{eq:unitsum2}
\frac{4\pi | \O^\times | }{2l+1} \sum_{\phi \in \B_l} \phi(x)\phi(y) =  p_l(\langle  x , y \rangle ) N(x,1,0) + O( l^{-1/8} ).
\end{equation}
Taking into account that $N(x,1,0) \geq 2$ and the value of $p_l(\langle  x , y \rangle )$ is asymptotically a negative constant, the bound \eqref{eq:linfty} for $\phi(x)$ and $\phi(y)$  implies that \eqref{eq:unitsum2} can hold for large $l$ only if at least $c(\e) l ^{1/6 - 2\e}$ of the $\phi \in \B_l$ give a negative contribution to the left hand side of \eqref{eq:unitsum2}, i.e. have different signs at $x$ and $y$.
\end{rmk}

\section{Proof of Theorem \ref{totalzeros} and Corollary \ref{zeroscoro}}
By use of Lemma \ref{lemma:euler}, Theorem \ref{totalzeros} will follow from the following Proposition \ref{totalzerosprop}. Corollary \ref{zeroscoro} follows directly from Theorem \ref{totalzeros} along with the dimension estimate for $\H_l^{\O^\times}$ given in \eqref{eq:dimension}.
\begin{prop}\label{totalzerosprop}
Let $l$ be even and $\B_l$ denote an orthonormal basis for $\H^{\O^\times}_l$ consisting of Hecke eigenfunctions. Fix one of the equators $E_j \equiv \{ (x_1,x_2,x_3) \in S^2 : x_j = 0 \}$, $j = 1,2,3$. Any $\varphi \in \B_l$ has finitely many zeros on $E_j$ and the members of $\B_l$ have a total number of zeros on $E_j$ which, for any $\e >0$, is bounded below by the formula
\begin{equation}
\sum_{\varphi \in \B_l} | \varphi^{-1}(\{0\}) \cap E_j | \gg_\e l^{5/4 - \e}.
\end{equation}
\end{prop}
To prove Proposition \ref{totalzerosprop} we will convert Proposition \ref{prop:signchange} into an estimate for the total number of zeros of Hecke eigenfunctions on equators $E_j \equiv \{ (x_1 , x_2, x_3) \in S^2 : x_j = 0 \}$. This is done by finding sufficiently many disjoint intervals in the equator where sign changes are provided by Proposition \ref{prop:signchange}. We will prove Proposition \ref{totalzerosprop} in the case when the equator is $E_3$. Recall that $l$ is even and $\B_l$ denotes an orthonormal basis of Hecke eigenfunctions for $\H_l^{\O^\times}$.

\begin{proof}[Proof of Proposition \ref{totalzerosprop}]
Let $\k > 0$ be an absolute constant (to be chosen small enough) and consider the set of points
\begin{equation}
S_l \equiv \left\{ \frac{(a,b,0)}{\sqrt{a^2 + b^2}} : a, b \in \Z , \: (a,b) = 1 , \: 0 < a, b < \sqrt{\k l / 2}  \right\} \subset E_3 .
\end{equation}
For each $x \in S_l$ choose $z(x)$ to be one of the points of distance $W_l$ from $x$. For any $x \in  S_l$ it is clear that $h(x) \leq a^2 + b^2 < \k l$. In particular $h(x)$ is eventually $\leq l^{3/2}$ so Proposition \ref{prop:signchange} gives that $\gg_\e l^{1/4 - \e}$ of the elements of $\B_l$ sign change between $x$ and $z(x)$.

We can ensure the intervals $[x , z(x)]$ and $[y, z(y)]$ are disjoint for $x \neq y \in S_l$ as follows. If $x = \frac{(a, b , 0 )}{\sqrt{a^2 + b ^2}}$ and $y = \frac{(c, d , 0 )}{\sqrt{c^2 + d ^2}}$ are distinct points in $S_l$ then 
 $\langle x, y \rangle \neq 1$ and so
\begin{equation}
\langle x , y \rangle^2 \leq 1 - \frac{1}{(a^2 + b^2)(c^2 + d^2)} 
\end{equation}
giving
\begin{equation} 
d(x,y) \gg \frac{1}{\sqrt {a^2 + b^2}\sqrt{c^2+d^2}} > \frac{1}{\k l} .
\end{equation}
We now choose $\k$ small enough so that
\begin{equation}
d(x, y) \geq 3 W_l
\end{equation}
and the intervals $[x , z(x)]$ and $[y, z(y)]$ are disjoint for any distinct $x , y \in S_l$.
We have the estimate for the size of $S_l$
\begin{equation}
| S_l | = \frac{6}{\pi^2}  \frac{ \k l }{2} + o(l) \gg l .
\end{equation}
Applying the intermediate value theorem to each of the intervals $[x , z(x)]$ we count per interval $\gg_\e l^{1/4 - \e}$ zeros in $[x, z(x)]$ of members of $\B_l$. There are $\gg l$ such disjoint intervals so in total we count $\gg_\e l^{5/4 - \e}$ zeros of members of $\B_l$ on $E_3$. 
\end{proof}
This concludes the proof of our `on average' result.
%%%%%%%%%%%%%%%%%%%%% ADELIZATION
\section{Adelization}\label{adelization}
Let $\AA$ denote the $\Q$-algebra of adeles, $\AA_{f}$ the finite adeles. Write $B$ for the Hamilton quaternions defined over $\Q$ so that for any $\Q$-algebra $R$ we define 
\begin{equation}
B(R)\equiv B(\Q) \otimes_\Q R .
\end{equation}
We view $B^\times$ as a algebraic group over $\Q$, and will consider in particular the adele group $B^\times(\AA)$. Let $\O$ denote the maximal order in $B(\Q)$, which has class number $h(\O) = 1$. By $\widehat{\O}$ we mean the closure of $\O$ in $B( \AA_{f} )$. There is an isomorphism
\begin{equation}\label{eq:adelicquaternion}
B^\times(\Q) \backslash B^\times(\AA)/ \widehat{\O}^\times \cong \O^\times \backslash B^\times(\R).
\end{equation}
The inverse of this map can be given explicitly by
\begin{align*}
&L : \O^\times \backslash B^\times(\R) \to B^\times(\Q) \backslash B^\times(\AA_{\Q}) / \widehat{\O}^\times , \\
&L : \O^\times x_\infty \mapsto B^\times(\Q) (1,1,1,\ldots , x_\infty) \widehat{\O}^\times .
\end{align*}

We choose $\mathbf{k} \in S^2$ as our north pole. The group $B^\times(\R)$ acts on $S^2$ by 
\begin{equation}
g.(x \mathbf{i} + y \mathbf{j} + z \mathbf{k}) = g (x \mathbf{i} + y \mathbf{j} + z \mathbf{k}) \bar{g} / n(g) , \quad x^2 +y^2 +z^2 = 1.
\end{equation}
Let $K_\infty = \mathrm{Stab}_{ B^\times(\R)}(\mathbf{k}) \cap n^{-1}(1)$ consist of the norm 1 elements in  $B^\times(\R)$ which fix $\mathbf{k}$. We can lift Laplacian eigenfunctions $\vp$ from $\O^\times \backslash S^2$ to Casimir eigenfunctions on $\O^\times \backslash B^\times(\R)$ by
\begin{equation}
\hat{\vp}(g) = \vp(g.\mathbf{k}) .
\end{equation}
If $\vp$ is an eigenfunction for all the Hecke operators on $L^2(S^2)$ then $L_* \hat{\vp}$, as a function on $ B^\times(\Q) \backslash B^\times(\AA)$, generates an irreducible cuspidal automorphic representation of $B^\times(\AA)$ which has trivial central character. The representation associated to $\vp$ in this way is denoted $\rho = \rho_\vp$ and $L_* \hat{\vp}$ is the unique $\widehat{\O}^\times K_\infty$ invariant vector in $\rho$. We have the restricted tensor product decomposition for $\rho$ over places $\nu$ of $\Q$
\begin{equation}
\rho \cong \bigotimes_{\nu} \rho_\nu .
\end{equation}
The local representation $\rho_\nu$ is ramified only at $2$ and $\infty$. We know that $\rho_2$ is the trivial representation or a character, and $\rho_\infty$ is the $l$th irreducible representation of $B^\times(\R) / \R \cong SO(3)$.

Now let $E = \Q(\mathbf{i})$, $\O_E = \Z[\mathbf{i}]$ the maximal order of $E$, and consider the embedding of normed $\Q$-algebras
\begin{equation}
\i : \Q(\mathbf{i}) \to B(\Q) , \quad \mathbf{i} \mapsto \mathbf{i}.
\end{equation}
We have then the optimality condition
\begin{equation}\label{eq:optimality}
\i(\O_E) = \i(E) \cap \O .
\end{equation}
Analogously to before there is an adelic isomorphism
\begin{equation}\label{eq:adelicabelian}
E^\times \backslash \AA_E^\times / \widehat{\O_E}^\times \cong \O_E^\times \backslash \C^\times 
\end{equation}
which follows from $E$ having class number $h(E) = 1$. The inclusion
\begin{equation}
E^\times \hookrightarrow B^\times(\Q)
\end{equation}
induces an embedding
\begin{equation}
E^\times \backslash \AA_E^\times \to B^\times(\Q) \backslash B^\times(\AA)
\end{equation}
and because of the optimality condition \eqref{eq:optimality} an embedding
\begin{equation}
E^\times \backslash  \AA_E^\times / \widehat{\O_E}^\times \to B^\times(\Q) \backslash  B^\times(\AA) / \widehat{\O}^\times .
\end{equation}
By using our previous adelic isomorphisms this gives an embedding
\begin{equation}
\O_E^\times \backslash \C^\times \cong  E^\times \backslash \AA_E^\times / \widehat{\O_E}^\times \to  B^\times (\Q)\backslash  B^\times (\AA)/ \widehat{\O}^\times  \cong \O^\times \backslash B^\times(\R) ,
\end{equation}
and this map agrees exactly with
\begin{equation}
\O_E^ \times \backslash \C^\times  \to \O^\times \backslash B^\times(\R) 
\end{equation}
induced directly from optimality \eqref{eq:optimality} and the inclusion $E \otimes \R \to B(\R)$ at the archimedean place, without any adelization. 

We have discussed lifting in the quaternion setting, now we discuss lifting in the abelian case. It will be sufficient for us to consider functions on $\C$
\begin{equation}
\omega_n(z)  = (z / \bar{z})^n , \quad n \in \frac{1}{2}\Z.
\end{equation}
In the special case when $n$ is an even integer, $\omega_n$ is invariant under $\O_E^\times$ and can be viewed as a character of $\O_E^\times \backslash \C$. It can therefore be lifted to a character $\Omega_n$ of the idele class group $E^\times \backslash \AA_E^\times $ which is trivial on $\widehat{\O_E}^\times$ via the isomorphism \eqref{eq:adelicabelian}. As $\Q$ has class number $h(\Q) = 1$ it follows that the restriction of $\Omega_n$ to $\AA^\times$ is trivial. It is clear that the Hecke character $\Omega_n$ constructed in this way is unramified at all the finite places of $E$ and of infinite order when $n \neq 0$.

%%%%%%%%%%%%%%%%%%%%%%%%%%%%%%%%%%%%%%%%%%%%%%%%%%%%%%%%
%%%%%%%%%%%%%%%% WALDSPURGER / JACQUET-CHEN / MARTIN-WHITEHOUSE %%%%%%%%%%%%%
%%%%%%%%%%%%%%%%%%%%%%%%%%%%%%%%%%%%%%%%%%%%%%%%%%%%%%%%
\section{Period formulae}\label{periodsection}
 Now we recall important work of Waldspurger \cite{WALD}, Jacquet and Chen \cite{JC} and Martin and Whitehouse \cite{MW} which will allow us to calculate the Fourier coefficients of automorphic forms in terms of central values of Rankin-Selberg $L$-functions. Our discussion and notation follow \cite{JC} and \cite{MW}.
 
Let $E = \Q(\mathbf{i})$ as before and $\pi$ denote a cuspidal automorphic representation of $\GL_2(\AA)$ which has trivial central character. Let $\Omega$ be a unitary Hecke character of $\AA_E^\times / E^\times$  whose restriction to $\AA^\times$ is trivial. For our application we need to assume that $\Omega$ and $\pi$ have disjoint ramification. 
 
We will deal with an irreducible unitary automorphic representation $\pi^B$ of $B^\times(\AA)$ associated to $\pi$ by the Jacquet-Langlands correspondence. In our case, $\pi^B = \rho_\vp$ for $\vp$ an automorphic form on $\O^\times \backslash S^2$, and $\Omega = \Omega_n$. As $\Omega_n$ is ramified nowhere the ramification of $\Omega_n$ and $\pi$ is trivially disjoint. The central characters of $\Omega_n$ and $\rho_\vp$ are both trivial.
 
 The object of study in this section is the period integral
 \begin{equation}
 P^B(\phi) \equiv \int_{E^\times \AA^\times \backslash \AA_E^\times} \phi(t) \Omega^{-1}(t) dt
 \end{equation}
 defined for all $\phi \in \pi^B$ via the identification of $E^\times \backslash \AA_E^\times$ with a subset of $B^\times(\Q) \backslash B^\times(\AA)$ and the matching of central characters. Let us motivate the study of $P^B$. Take $\phi$ to be the spherical vector in $\rho_\vp$, which corresonds to the lift of the spherical harmonic $\vp$. Up to choice of normalization for Haar measures, by using the adelic isomorphisms of Section \ref{adelization} we have, modulo positive constants,
 \begin{align}
 P^B(\phi) &= \int_{\theta \in [0,2 \pi)} \vp( (\cos(\theta) + \sin(\theta) \mathbf{i}) . \mathbf{k} ) e^{-i 2 n \theta} d \theta \\
 &= \int_{\theta \in [0,2 \pi)} \vp( e^{\mathbf{i} 2\theta} \mathbf{k} ) e^{-i 2 n \theta} d \theta =\int_{\theta \in [0,2 \pi)} \vp( e^{\mathbf{i} \theta} \mathbf{k} ) e^{-i n \theta} d \theta
 \end{align}
 so that $P^B(\phi)$ is a multiple (not depending on $\vp$) of the $n$th Fourier coefficient of the classical automorphic form $\vp$ restricted to the geodesic $\{  e^{\mathbf{i} \theta} \mathbf{k} \}$.
 
Now write $\Pi = BC_{E/\Q}(\pi)$ for the base change of $\pi$ to an automorphic representation of $\GL_2(\AA_E)$, which is cuspidal unless $\pi$ is dihedral.  In \cite{WALD}, Waldspurger proved the existence of a formula relating $|P^B(\phi)|^2$ to the central value $L(1/2, \Pi \otimes \Omega)$. Certain factors in this formula are not as explicit as we would like for our application, although they could undoubtedly be calculated. A formula of the same flavor, which does have explicit factors, was proven by Jacquet and Chen \cite{JC} in the case when $\pi$ is not dihedral, and by Martin and Whitehouse \cite{MW} in the dihedral case. The work of Jacquet and Chen utilizes the relative trace formula and the factorization of a certain distribution which we will describe now. We note that in all the formulae we are aware of, only the square of the period integral is obtained. It would be nice to have the phase, but we do not know how to get it at present.
 
The distribution we need to consider is
\begin{equation}
J_{\pi^B} : C^\infty_c( B^\times(\AA) ) \to \C , \quad f \mapsto \sum_{\phi} \int [\pi^B(f)\phi](t)\Omega(t)^{-1} dt \overline { \int \phi(t)\Omega(t)^{-1} dt },
\end{equation}
where summation is over an orthonormal basis of $\pi^B$ and
\begin{equation}
 \pi^B(f)\phi = \int_{B^\times(\AA)} f(g) [\pi^B(g) \phi] dg .
\end{equation} 
Let $S_0$ be a finite set of places of $\Q$ containing the infinite place and all places where $\pi^B$ or $\Omega$ is ramified, which for us will be $S_0 = \{2,\infty\}$. Denote by $S$ all the places of $E$ lying over $S_0$. For any $f = \left( \prod_{\nu \in S_0} f_\nu \right) f^{S_0}$, where $f^{S_0}$ is the characteristic function of a maximal compact subgroup of $B^\times(\AA^{S_0})$, the distribution $J_{\pi^B}$ factorizes at $f$. 
We will take $f$ such that $\pi^B(f)$ is orthogonal projection onto the span of the $L^2(B^\times(\AA))$-normalized spherical vector $\varphi$ in $\pi^B$. By abuse of notation $\vp$ is the lift of a Hecke eigenfunction $\vp$ on $S^2$. Then
\begin{equation}\label{eq:expression1}
J_{\pi^B}(f) = |P^B(\varphi)|^2.
\end{equation}
On the other hand we have the factorization (\cite[Theorem 2]{JC}, \cite[Appendix]{MW})
\begin{equation}\label{eq:expression2}
J_{\pi^B}(f) = \frac{1}{2}\prod_{\nu \in S_0} \tilde{J}_{\pi^B}(f_\nu) \times \left( \prod_{\substack{\nu \in S_0 ,\\ \mathrm{inert}}} \e(1,\eta_\nu, \psi_\nu) 2L(0,\eta_\nu) \right) \times \frac{ L_{S_0}(1,\eta)L^S(1/2 , \Pi \otimes \Omega ) }{ L^{S_0}(1,\pi,\mathrm{Ad} )} ,
\end{equation}
where $\eta$ is the quadratic Hecke character of $\AA^\times$ associated to $E$.
The terms in the central parenthesis have not been defined (see \cite{JC}), but they are positive and depend only on the ramification of the representations and $E$. The same is true for $L_{S_0}(1,\eta)$ (see our notation index for conventions regarding $L$-functions). The terms $ \tilde{J}_{\pi^B}(f_\nu)$ are local factors of the distribution which need to be calculated per ramification scenario as in \cite{JC}. With these observations and the two expressions \eqref{eq:expression1} and \eqref{eq:expression2} for $J_{\pi^B}(f)$ we can write
\begin{equation}\label{eq:prefourier}
|P^B(\varphi)|^2 = C_{\mathrm{ram}(\Omega),\mathrm{ram}(\pi), E}\prod_{\nu \in S_0} \tilde{J}_{\pi^B}(f_\nu) \times \frac{ L^S(1/2 , \Pi \otimes \Omega ) }{ L^{S_0}(1,\pi,\mathrm{Ad} )} .
\end{equation}
When we apply this equation we will want it to be in terms of the finite $L$-functions $L(1/2 , \Pi \otimes \Omega ) $ and $ L(1,\pi,\mathrm{Ad} )$, and we will want to know the local factors $ \tilde{J}_{\pi^B}(f_\nu) $ modulo positive values which are bounded above and away from zero. In the next Lemma we obtain the desired relation.
%%%%%%%%%%%%%%% LEMMA
\begin{lemma}\label{fourier1} Let $\vp \in \H_l^{\O^\times}$ be an eigenfunction of the Hecke operators and the spherical Laplacian with $\|\vp\|_{L^2(S^2)} =1 $. With the previous notation, let $\pi$ be the Jacquet-Langlands transfer of $\rho = \rho_\vp$ to an automorphic representation of $\GL_2(\AA)$ and $\Pi = BC_{E /\Q}(\pi)$. Then we have
\begin{equation}
\left| \int_{\theta \in [0,2 \pi)} \vp( e^{\mathbf{i} \theta} \mathbf{k} ) e^{-i n \theta} d \theta \right|^2 \asymp \frac{ Y^n_l(0)^2 }{2l+1}\frac{ L(1/2 , \Pi \otimes \Omega_n ) }{ L(1,\pi,\mathrm{Ad} )} ,
\end{equation}
where $Y^n_l$ is the associated Legendre polynomial $P^n_l$ normalized so as to have \\$\| Y^n_l( \langle \bullet, \mathbf{k} \rangle ) \|_{L^2(S^2)} = 1$.
\end{lemma}
%%%%%%%%%%%%%% PROOF
\begin{proof}
Given that $\| \vp \|_{L^2(S^2)} = 1$,
\begin{equation} \label{eq:start}
 \left| \int_{\theta \in [0,2 \pi)} \vp( e^{\mathbf{i} \theta} \mathbf{k} ) e^{-i n \theta} d \theta \right|^2 \asymp \prod_{\nu \in S_0} \tilde{J}_{\rho}(f_\nu) \times \frac{ L^S(1/2 , \Pi \otimes \Omega_n ) }{ L^{S_0}(1,\pi,\mathrm{Ad} )} 
 \end{equation}
 follows from the previous discussion, particularly equation \eqref{eq:prefourier}. Here $S_0 = \{2 , \infty \}$ so it remains to calculate $\tilde{J}_{\rho}(f_\nu)$ at these places as well as the Euler factor at $2$ for $L(1,\pi,\mathrm{Ad} )$ and the factor at $(1+i)$ for $ L(1/2 , \Pi \otimes \Omega_n )$. Our $f$ agrees with that of Martin and Whitehouse \cite{MW} at $2$ and they calculate that in our case ($\pi$ ramified at $2$ and $B$ non-split),  
 \begin{equation}
 \tilde{J}_{\rho}(f_2) = 1. 
 \end{equation}
 For the factor $ \tilde{J}_{\rho}(f_\infty)$ we follow \cite[Section 5.1]{JC}. We pick as a model for $\rho_\infty$ the space $\H_l$ of homogeneous harmonic polynomials of degree $l$ on $\R^3$ restricted to the sphere. We need to find unit vectors $e_T$, $e'_T$ such that
 \begin{equation}
 \rho_\infty(t) e_T = \Omega_n^{-1}(t) e_T ,\quad  \rho_\infty(t) e'_{T} = \Omega_n(t) e'_{T}
 \end{equation}
 for $t$ in $\R^{\times} \backslash \C^\times$.  In our model these vectors are given by
 \begin{equation}
 e_T(v) = Y^n_l( \langle v, \mathbf{i} \rangle ) e^{- i n \theta} , \quad e'_{T}(v) = Y^n_l( \langle v, \mathbf{i} \rangle ) e^{i n \theta},
 \end{equation}
 where $\theta$ is an angle around the axis of $\mathbf{i}$ such that $\theta(\mathbf{k}) = 0$. We have then
 \begin{equation}
 \tilde{J}_{\rho}(f_\infty) =  \langle \rho_\infty(f_\infty) e_T , e_{T'} \rangle .
 \end{equation}
 We chose $f_\infty$ so that $\rho_\infty(f_\infty)$ was orthogonal projection onto the $\mathbf{k}$-spherical vector. We can therefore write
 \begin{equation}
  [\rho_\infty(f_\infty) e_T](y) = \left( \int_{x \in S^2} Y^0_l(\langle \mathbf{k} , x \rangle) Y^n_l( \langle x, \mathbf{i} \rangle )e^{-in\theta} dx \right)    Y^0_l(\langle \mathbf{k} , y \rangle) .
  \end{equation}
 Using that the kernel $\sqrt{2l+1}   Y^0_l(\langle \bullet , \bullet \rangle)$ is reproducing gives then
 \begin{equation}
   [\rho_\infty(f_\infty) e_T](y) = \frac{1}{\sqrt{2l+1}} Y^n_l(0)   Y^0_l(\langle \mathbf{k} , y \rangle) .
 \end{equation}
Now we use the reproducing property again to calculate
\begin{align*} 
\langle \rho_\infty(f_\infty) e_T , e_{T'} \rangle &=\frac{Y^n_l(0) }{\sqrt{2l+1}}  \int_{y \in S^2}Y^0_l(\langle \mathbf{k} , y \rangle)Y^n_l( \langle y, \mathbf{i} \rangle ) e^{-i n \theta} dy \\
&= \frac{Y^n_l(0)^2}{{2l+1}},
\end{align*}
which gives the value for  $\tilde{J}_{\rho}(f_\infty)$. 

The local $L$-factors at 2 can be calculated by local Langlands. For background on the following discussion see the article of Tate \cite{TATE}. Since $\pi$ has conductor 2 and trivial central character it follows that $\pi_2$ and hence $\Pi_{(1 + i )}$ are twists of the Steinberg representation. Then $\Pi_{(1 + i)}$ corresponds to a Weil-Deligne parameter consisting of a representation $\tau : W_{E_{(1+i)}} \to \GL_2(\C)$ of the Weil group of $E_{(1 +i )}$ together with the nilpotent endomorphism of $\C^2$
\begin{equation}
N = \left( \begin{array}{cc}
0 & 1 \\
0 & 0 \end{array} \right).
\end{equation}
The geometric Frobenius $\Phi$ maps under $\tau$ to
\begin{equation}
\tau(\Phi) = \left( \begin{array}{cc}
\chi(\Phi) 2^{-1/2} & 0 \\
0 & \chi(\Phi) 2^{1/2} \end{array} \right)
\end{equation}
where $\chi$ is the character parameterizing the twist of Steinberg. Then the Weil-Deligne parameter for $\Pi_{(1 + i )} \otimes (\Omega_n)_{(1+i)}$ consists of the same $N$ together with a representation $\tau'$ of the Weil group under which $\Phi$ maps to
\begin{equation}
\tau'(\Phi) = \left( \begin{array}{cc}
\chi(\Phi) \alpha_n(\Phi) 2^{-1/2} & 0 \\
0 & \chi(\Phi) \alpha_n(\Phi) 2^{1/2} \end{array} \right)
\end{equation}
where $\a_n$ corresponds to $(\Omega_n)_{(1+i)}$ via local class field theory. The local $L$-factor is then given by
\begin{align}
L_{(1+i)}(1/2 , \Pi \otimes \Omega_n ) &= \det( 1 - \tau'(\Phi)\lvert_{\ker{N}} 2^{-1/2} )^{-1} \\
&=  \left( 1 - \frac{\chi(\Phi)\alpha_n(\Phi)}{ 2} \right)^{-1}.
\end{align}
As the characters are unitary we have
\begin{equation}
L_{(1+i)}(1/2 , \Pi \otimes \Omega_n ) \asymp 1.
\end{equation}
Similarly the Weil-Deligne parameter for $\pi_2$ consists of a representation $\sigma : W_{\Q_2} \to \GL_2(\C)$ with the same $N$ as before and geometric Frobenius $\Psi$ mapping to 
\begin{equation} 
\sigma(\Psi) = \left( \begin{array}{cc}
\chi(\Psi) 2^{-1/2} & 0 \\
0 & \chi(\Psi) 2^{1/2} \end{array} \right)
\end{equation}
for some unitary $\chi$. Now let $e_1 , e_2$ denote the standard basis for $\C^2$ and $\tilde{e}_1 , \tilde{e}_2$ denote the dual basis. Then the adjoint representation of $\GL_2(\C)$ can be realized as the 3 dimensional invariant subspace of $\C^2 \otimes (\C^{2})^*$ under $\mathrm{standard} \otimes \widetilde{\mathrm{standard}}$ given by
\begin{equation}\label{eq:Vbasis}
V = \langle e_1 \otimes \tilde{e}_1 - e_2 \otimes \tilde{e}_2 , e_1 \otimes \tilde{e}_2, e_2 \otimes \tilde{e}_1 \rangle.
\end{equation} 
The three dimensional Weil-Deligne representation for $\mathrm{Ad}(\pi_2)$ is given by the restriction of $\sigma \otimes \tilde{\sigma}$ to $V$ together with the nilpotent element $N_{\mathrm{Ad}}$ which is the restriction of $N \otimes 1 - 1 \otimes N^T$ to $V$. Therefore with respect to the basis in \eqref{eq:Vbasis} we have
\begin{equation}
\mathrm{Ad}(\sigma)(\Psi)  = \left( \begin{array}{ccc}
1 & 0 & 0 \\
0 & 2^{-1} & 0 \\
0 & 0 & 2 \end{array} \right) \quad
N_{\mathrm{Ad}}  = \left( \begin{array}{ccc}
0 & 0 & 1 \\
-2 & 0 & 0 \\
0 & 0 & 0 \end{array} \right).
\end{equation}
Then
\begin{equation}
L_2(1 , \pi , \mathrm{Ad}) = \det( 1 - \mathrm{Ad}(\sigma)(\Psi)\lvert_{\ker N_{\mathrm{Ad} }} 2^{-1} )^{-1} = 4/3.
\end{equation}

Substituting our calculations into \eqref{eq:start} gives the desired result.
\end{proof}
There is a slightly more natural way to write Lemma \ref{fourier1} which we state now.

\begin{lemma}\label{fourier2} With all our previous notation, suppose that $\vp \in \H_l^{\O^\times}$ is an eigenfunction of the Hecke operators and the spherical Laplacian which has the ultraspherical expansion around $\mathbf{i}$
\begin{equation}
\vp( x ) = \sum_{|m| \leq l} a_m Y^m_l( \langle x , \mathbf{i} \rangle ) e^{i m \theta(x) } , \quad \sum_{|m| \leq l} |a_m|^2 = 1 .
\end{equation}
Suppose that $l$ is even. Then $a_m$ is zero unless $m \equiv 0 \bmod 2$ and in this case 
\begin{equation}
|a_m|^2 \asymp \frac{ 1 }{2l+1}\frac{ L(1/2 , \Pi \otimes \Omega_m ) }{ L(1,\pi,\mathrm{Ad} )} .
\end{equation}
\end{lemma}
\begin{proof}
The square of the absolute value of the $m$th Fourier coefficient of the function $\vp\lvert_{\{\langle x , \mathbf{i} \rangle= 0 \}}$ is given by Lemma \ref{fourier1}. On the other hand the square of the absolute value of the $m$th Fourier coefficient of $\vp\lvert_{\{\langle x , \mathbf{i} \rangle = 0 \} }$ is
\begin{equation}
|a_m|^2 Y^m_l(0)^2.
\end{equation}
Comparing the two expressions, the factors of $Y^m_l(0)^2$ cancel when they are not zero - this is true when $m \equiv l \equiv 0 \bmod 2$. 

It remains to note that if $m$ is odd then $a_m = 0$. This follows from the invariance properties of $\vp$ - we discuss this in more detail in the proof of Lemma \ref{restriction}.
\end{proof}

\section{The subconvexity hypothesis} In this section we introduce a hypothesis which controls the size of the $L$-functions which are related to Fourier/ultraspherical coefficients. For a general automorphic representation $\pi$ of $\GL_n(\AA_F)$, $F$ a number field, Iwaniec and Sarnak have defined the \textit{analytic conductor} $C(\pi)$. This $C(\pi)$ is the product of the usual conductor and a parameter which measures the infinity type of the automorphic representation, and it turns out that $C(\pi)$ is the natural quantity by which to estimate the central value $L(1/2 , \pi)$.

In our setting the classical conductor does not vary, so estimates versus $C(\pi)$ are in terms of the archimedean (eigenvalue/weight) aspect. One expects a (relatively) easy estimate of the form
\begin{equation}
L(1/2 , \pi) \ll_{n, F, \e} C(\pi)^{1/4 + \e} .
\end{equation}
Improving over this bound is called the \textit{subconvexity problem} and has been accomplished in some cases. For now we record the important result of Michel and Venkatesh \cite[Theorem 1.1]{MV}.
\begin{thm}[Michel, Venkatesh]\label{MV} There is an absolute constant $\d > 0$ such that: for $\pi$ an automorphic representation of $\GL_1(\AA_F)$ or $\GL_2(\AA_F)$ (with unitary central character), one has 
\begin{equation}
L(1/2 , \pi) \ll_F C(\pi)^{1/4 - \d}. 
\end{equation}
\end{thm}
In the setting of Lemma \ref{fourier2} we are interested in the $L$-value $L(1/2 , \Pi \otimes \Omega_m)$ (we will address $L(1 , \pi , \mathrm{Ad})$ momentarily). We know that $L(s , \Pi \otimes \Omega_m)$ is the automorphic $L$-function of an automorphic representation of $\GL_2(\AA_{E})$ (so Theorem \ref{MV} applies to the central value). The Gamma factors for $L(s , \Pi \otimes \Omega_m)$ in the case $|m| \leq l$ are \cite[pg. 169]{MW}
\begin{equation}
\G_{\C}\left(s + \frac{k-1}{2} + |m|\right)\G_{\C}\left(s + \frac{k-1}{2} - |m|\right),
\end{equation}
where $k$ is the weight of the holomorphic cusp form associated to $\pi$ and the complex Gamma function is
\begin{equation}
\G_\C(s) \equiv 2(2\pi)^{-s}\G(s) .
\end{equation}
Following the recipe for the analytic conductor given in \cite[pg. 207]{MV} we have
\begin{equation}
C(\Pi \otimes \Omega_m) = \left(2 + \frac{k-1}{2} + |m|\right)^2 \left(2 + \frac{k-1}{2} - |m|\right)^2 .
\end{equation}
Recall that $\pi$ is the Jacquet-Langlands transfer of a representation generated by a spherical harmonic of degree $l$, so we have $k = 2l+2$ and the analytic conductor can be written
\begin{equation}\label{eq:conductor}
C(\Pi \otimes \Omega_m) \asymp (5/2 + l + |m|)^2(5/2 + l - |m|)^2 .
\end{equation}

We will want to assume a general subconvexity hypothesis for these $L$-functions which we state now.
\\
\begin{description}
\item[Hypothesis SC($\delta$)] $\quad L(1/2 , \Pi \otimes \Omega_m) \ll_\delta C(\Pi \otimes \Omega_m)^{1/4 - \delta} $. \\
\end{description}
The ability to take $\delta$ arbitrarily close to $1/4$ in SC($\delta$) is a section of the Generalized Lindel\"{o}f hypothesis. We therefore consider only $\delta$ values in $(0,1/4)$. The result of Michel and Venkatesh says that SC($\delta$) is true for some $\delta \in (0,1/4)$.

It is well known that the values $L(1, \pi , \mathrm{Ad})$ satisfy
\begin{equation}
(\log k)^{-2} \ll L(1, \pi , \mathrm{Ad}) \ll ( \log k)^2 ,
\end{equation}
or keeping with our $l$ parameter
\begin{equation}
(\log l)^{-2} \ll L(1, \pi , \mathrm{Ad}) \ll (\log l)^2 .
\end{equation}

\begin{rmk}Given $\vp \in \H_l^{\O^\times}$ a Hecke eigenfunction with 
\begin{equation}
\vp( x ) = \sum_{|m| \leq l} a_m Y^m_l( \langle x , \mathbf{i} \rangle ) e^{i m \theta(x) } , \quad \sum_{|m| \leq l} |a_m|^2 = 1 
\end{equation}
we have by Lemma \ref{fourier2} that when $m$ and $l$ are even 
\begin{equation}
\frac{ 1 }{2l+1}\frac{ L(1/2 , \Pi \otimes \Omega_m ) }{ L(1,\pi,\mathrm{Ad} )} \ll |a_m|^2 \leq 1.
\end{equation}
Thus the inequality
\begin{equation}
L(1/2 , \Pi \otimes \Omega_m ) \ll l \: L(1,\pi,\mathrm{Ad}) \ll l (\log l)^2
\end{equation}
gives a convexity bound for the $L$-value $L(1/2 , \Pi \otimes \Omega_m )$ when $|m|$ is not too close to $l$.
\end{rmk}

To conclude this section we note that using the bounds for $L(1 , \pi , \mathrm{Ad})$ along with SC($\delta$) in Lemma \ref{fourier2} gives the upper bound for the ultraspherical coefficient
\begin{equation}\label{eq:bounda}
|a_m| \ll \frac{\log l}{\sqrt{l}} (1 + l + |m|)^{1/4 -\delta}(1 + l - |m|)^{1/4 - \delta}
\end{equation}
where we replaced $5/2$ with $1$ in the expression \eqref{eq:conductor}  for $C(\Pi \otimes \Omega_m)$ for simplicity.

%%%%%%%%%%%%%%%%%%%%%%%%%%%%%%%%%%%%%%%
\section{Values of ultraspherical polynomials and an $L^2$ restriction lower bound}
Notice that the value $Y^m_l(0)$ appears in Lemma \ref{fourier1}. This is given by VanderKam \cite[pg. 338]{VAND} as
\begin{equation}
Y^m_l(0) = \frac{\sqrt{2l+1}}{2^l} \sqrt{ \binom{l+ |m|}{\frac{1}{2}(l+|m|)}\binom{l-|m|}{\frac{1}{2}(l-|m|)} }
\end{equation}
where the combinatorial symbols are understood to vanish when they have non integer arguments.
Using the Stirling bounds
\begin{equation}
n^{(n+1)/2} e^{-n} \ll n! \ll n^{(n+1)/2} e^{-n}
\end{equation}
to estimate the binomial coefficients (with a little adjustment for $|m| = l$) gives
\begin{equation}\label{eq:boundY}
 Y^m_l(0) \asymp \frac{\sqrt{l}}{(1 + l + |m|)^{1/4}(1+l - |m|)^{1/4}}.
\end{equation}
In particular
\begin{equation}\label{eq:valueboundbelow}
Y^m_l(0) \gg 1
\end{equation}
which gives an easy $L^2$ restriction lower bound in the following Lemma.

\begin{lemma}\label{restriction}
Let $\g$ denote the geodesic $\{ \langle x , \mathbf{i} \rangle = 0 \}$ in $S^2$. When $l$ is even and $\vp \in \H_l^{\O^\times}$ we have
\begin{equation}
\| \vp\lvert_\g \|_{L^2(\g)} \gg \| \vp \|_{L^2(S^2)}.
\end{equation}
\end{lemma}
\begin{proof}
When $l$ is even and $\vp \in \H_l^{\O^\times}$ we have noted in Section \ref{euler} that $\vp$ is invariant under the involution of $S^2$
\begin{equation}
(x,y,z) \mapsto (-x , y , z).
\end{equation}
It follows that in the expansion of $\vp$ about $\mathbf{i}$ the only nonzero coefficients correspond to even $m$, i.e.
\begin{equation}
\vp( x ) = \sum_{\substack {|m| \leq l \\ m \equiv 0 \bmod 2}} a_m Y^m_l( \langle x , \mathbf{i} \rangle ) e^{i m \theta(x) }.
\end{equation}
By Parseval, the $L^2$ norm of $\vp$ restricted to $\g$ is
\begin{equation}
\| \vp\lvert_\g \|_{L^2(\g)} = \sum_{m \equiv 0 \bmod 2} |a_m|^2 Y^m_l(0)^2 ,
\end{equation}
and using the bound below for $Y^m_l(0)$ from \eqref{eq:valueboundbelow} gives the required
\begin{equation}
\| \vp \lvert_\g \|_{L^2(\g)} \gg \sum_{m \equiv 0 \bmod 2 } |a_m|^2 = \| \vp \|_{L^2(S^2)}.
\end{equation}
\end{proof} 

%%%%%%%%%%%%%%%%%%%%%%%%%%%%%%%%%%%%%%%%%%%%%%%%%%
%%%%%%%%%%%%%% PROOF OF CONDITIONAL THEOREM
%%%%%%%%%%%%%%%%%%%%%%%%%%%%%%%%%%%%%%%%%%%%%%%%%%

\section{Proof of Theorem \ref{uniformzeros}}
Just as in our proof of Theorem \ref{totalzeros}, Theorem \ref{uniformzeros} follows from Lemma \ref{lemma:euler} together with the following Proposition \ref{conditionalprop}.
\begin{prop}\label{conditionalprop}
Assuming the generalized Lindel\"{o}f hypothesis, then for any $\e > 0$, $\vp$ has $\gg_\e l^{1/12 - \e}$ zeros on the geodesic
\begin{equation}
\{ e^{\mathbf{i}\theta}\mathbf{k} \} \subset S^2 = \{ x\mathbf{i} + y \mathbf{j} + z \mathbf{k} : x^2 + y^2 + z^2 = 1 \}
\end{equation}
when $l$ is even and $\vp$ is a Hecke eigenfunction in $\H^{\O^\times}_l$.
\end{prop}
\begin{proof}
Let $l$ be a positive even integer and $\vp \in \H_l^{\O^\times}$ a Hecke eigenfunction. We assume that $\| \vp \|_{L^2(S^2)} = 1$. We denote by $\g$ the geodesic 
\begin{equation}
\{ (0,y,z) \in S^2 \} = \{ v \in S^2 : \langle v , \mathbf{i} \rangle = 0 \}.
\end{equation}
We divide $\g$ into intervals which have as endpoints roots of $\vp \lvert_\g$. That is, we write 
\begin{equation}
\g = \coprod_{i = 0}^{N-1} I_i, \quad I_i \equiv [(0,y_i,z_i),(0,y_{i+1},z_{i+1}) ]
\end{equation}
where the subscripts run mod $N$ and $\vp(0,y,z) = 0$ if and only if $(y,z) = (y_i,z_i)$ for some $i$. Moreover it is understood that the $(y_i , z_i)$ are in the correct order mod $N$. Then $\vp$ has $N$ zeros on $\g$.

It is clear then that $\vp$ has constant sign on the interior of each $I_i$. This means we can estimate $\| \vp\lvert_\g \|_{L^1(\g)}$ as
\begin{equation}
\| \vp\lvert_\g \|_{L^1(\g)} = \int_\g | \vp(x) | d\g(x) = \sum_{i=0}^{N-1} \left| \int_{I_i} \vp(x) d\g(x) \right|.
\end{equation}
If we can
\begin{enumerate}
\item Show that each of the contributions $ \left| \int_{I_i} \vp(x) d\g(x) \right|$ is small and
\item Show that $\| \vp\lvert_\g \|_{L^1(\g)}$ is large (in comparison)
\end{enumerate}
then it will follow that $N$ must be large. This is the approach of Ghosh, Reznikov and Sarnak in \cite{GRS} which we follow rather closely.

\textit{We now assume SC($\delta$) for some $\delta$ to be chosen in $(0,1/4)$.}

We write as before the ultraspherical expansion of $\vp$ around $\mathbf{i}$
\begin{equation}
\vp( x ) = \sum_{|m| \leq l} a_m Y^m_l( \langle x , \mathbf{i} \rangle ) e^{i m \theta(x) } ,
\end{equation}
so that if we define
\begin{equation}
b_m \equiv a_m Y^m_l(0)
\end{equation}
we have the Fourier expansion 
\begin{equation}
\vp\lvert_\g(e^{\mathbf{i}\theta}\mathbf{k}) = \sum_{|m| \leq l} b_m e^{im\theta}.
\end{equation}
As $\vp$ is $L^2$-normalized Lemma \ref{restriction} gives
\begin{equation}
\| \vp\lvert_\g \|_{L^2(\g)} \gg 1.
\end{equation} 
The bound for $|a_m|$ from \eqref{eq:bounda} together with the bound for $Y^m_l(0)$ in \eqref{eq:boundY} gives
\begin{equation}
| b_m | \ll \frac{\log l}{(1 + l + |m|)^\delta(1 + l - |m|)^{\delta}} .
\end{equation}
Another crucial input is Theorem \ref{vkinfinity} which is the $L^\infty$ bound of VanderKam \cite{VAND}
\begin{equation}\label{eq:linfinity}
\| \vp \|_{\infty} \ll_\e l^{5/12 + \e} .
\end{equation}

We fix an $i \in [0,N-1]$ and write for simplicity $I = I_i$. We will bound $\left| \int_{I} \vp(x) d\g(x) \right|$ in terms of $\| \vp\lvert_\g \|_{L^1(\g)}$. We have (by Plancherel)
\begin{align*}
\left| \int_{I} \vp(x) d\g(x) \right| &= \left| \sum_{|m| \leq l} b_m \left( \int_I e^{-im\theta} d\g(\theta) \right) \right|\\
&\leq \sum_{|m| \leq l} |b_m|\left| \int_I e^{-im\theta} d\g(\theta) \right|,
\end{align*}
and bounding the last integral in size by $1/(1+|m|)$ gives
\begin{equation}\label{eq:intbound}
\left| \int_{I} \vp(x) d\g(x) \right| \ll \sum_{|m| \leq l} |b_m| \frac{1}{1 + |m|}.
\end{equation}
We cut the range of the sum into two regions to be estimated separately:
\begin{equation}
\sum_{|m| \leq l} |b_m| \frac{1}{1 + |m|} = \sum_{|m| \leq l/2} |b_m| \frac{1}{1 + |m|} + \sum_{|m| > l/2}|b_m| \frac{1}{1 + |m|}.
\end{equation}
When $|m| \leq l/2$ we have
\begin{equation}
| b_m | \ll \frac{\log l}{(1 + l + |m|)^\delta(1 + l - |m|)^{\delta}} \ll \frac{\log l}{l^{2\delta}}.
\end{equation}
so that
\begin{equation}\label{eq:est10}
\sum_{|m| \leq l/2} |b_m| \frac{1}{1 + |m|} \ll (\log l) l^{-2\delta}\sum_{|m| \leq l/2} \frac{1}{1 + |m|} \ll (\log l)^2 l^{-2\delta}.
\end{equation}
In the other range we use Cauchy-Schwarz:
\begin{equation}
 \sum_{|m| > l/2}|b_m| \frac{1}{1 + |m|} \ll  \left(\sum_{|m| > l/2}|b_m|^2 \right)^{1/2} \left( \sum_{|m| > l/2}\frac{1}{(1 + |m|)^2} \right)^{1/2}
\end{equation}
and by Parseval
\begin{equation}
 \left(\sum_{|m| > l/2}|b_m|^2 \right)^{1/2} \leq \| \vp\lvert_\g \|_{L^2(\g)}
 \end{equation}
 giving
\begin{equation}\label{eq:est20} 
 \sum_{|m| > l/2}|b_m| \frac{1}{1 + |m|} \ll l^{-1/2}\| \vp\lvert_\g \|_{L^2(\g)}.
\end{equation}
Putting inequalities \eqref{eq:est10} and \eqref{eq:est20} together and recalling \eqref{eq:intbound} gives
\begin{equation}
\left| \int_{I} \vp(x) d\g(x) \right|  \ll (\log l)^2 l^{-2\delta} +  l^{-1/2}\| \vp\lvert_\g \|_{L^2(\g)}.
\end{equation} 
To proceed we use $\| \vp\lvert_\g \|_{L^2(\g)} \gg 1$ so that we have
\begin{equation}\label{eq:est30}
\left| \int_{I} \vp(x) d\g(x) \right| \ll \left((\log l)^2 l^{-2\delta} +  l^{-1/2} \right) \| \vp\lvert_\g \|^2_{L^2(\g)} \ll (\log l)^2 l^{-2\delta}  \| \vp\lvert_\g \|^2_{L^2(\g)} .
\end{equation}
Noting
\begin{equation}
\| \vp\lvert_\g \|^2_{L^2(\g)} \leq \| \vp\lvert_\g \|_{L^\infty(\g)} \| \vp\lvert_\g \|_{L^1(\g)},
\end{equation}
the $L^\infty$ bound \eqref{eq:linfinity} gives
\begin{equation}
\| \vp\lvert_\g \|^2_{L^2(\g)} \ll_\e l^{5/12 + \e} \| \vp\lvert_\g \|_{L^1(\g)}.
\end{equation}
Putting this into \eqref{eq:est30} gives
\begin{equation}
\left| \int_{I} \vp(x) d\g(x) \right| \ll_\e (\log l)^2 l^{5/12 + \e -2\delta}  \| \vp\lvert_\g \|_{L^1(\g)}. 
\end{equation}
We now assume the Lindel\"{o}f hypothesis for the relevant $L$-functions: taking $\delta$ arbitrarily close to $1/4$ gives
\begin{equation}\label{eq:est40}
\left| \int_{I} \vp(x) d\g(x) \right| \ll_\e l^{-1/12  + \e }  \| \vp\lvert_\g \|_{L^1(\g)}.
\end{equation}
As the contributions from the intervals $I_i$ paving the geodesic must make up the total $L^1$ norm, the estimate \eqref{eq:est40} implies that there must be $\gg_\e l^{1/12 - \e}$ intervals, hence $\gg_\e l^{1/12 - \e}$ zeros of $\vp$ on $\g$.
\end{proof}
%%%%%%%%%%%%%%%%%%%%%%%%%%%%%%%%%%%%%%%%%%%%%%%%%%%%%%%%%%%%%%%%%%%%%%%%%%%%%%%%%%%%%%%%%%%%%%%%%%%
%%%%%%%%%%%%%%%%%%%%%%%%%%%%%%%%%%%%%%%%%%%% NOTATION %%%%%%%%%%%%%%%%%%%%%%%%%%%%%%%%%%%%%%%%%%%%%
%%%%%%%%%%%%%%%%%%%%%%%%%%%%%%%%%%%%%%%%%%%%%%%%%%%%%%%%%%%%%%%%%%%%%%%%%%%%%%%%%%%%%%%%%%%%%%%%%%%
\section{Notation}
\subsection{Relations}
\begin{center} 
\begin{tabular}{l l }
$f \asymp g$ &means that $f \ll g$ and $g \ll f$.
\end{tabular}
\end{center}

\subsection{Symbols}
\begin{center}
\begin{tabular}{l l}
$\AA$ & The adele ring of $\Q$.\\
$\AA_E$ & The adele ring of $E$.\\
$B$ & Hamilton quaternions over $\Q$.\\
$\B_l$ & Orthonormal basis of Hecke eigenfunctions for $\H_l^{\O^\times}$.\\
$\H_l$ & Homogeneous harmonic polynomials of degree $l$ restricted to $S^2$. \\
$L(\bullet, \bullet)$ & The finite $L$-function (without $\G$ factors). \\
$L^S(\bullet, \bullet)$  & The $L$-function with the Euler factors at $S$ removed.  \\   
$L_S(\bullet, \bullet)$  & The $L$-function with only the Euler factors at $S$ present. \\ 
$\N(f)$ & The number of nodal domains of a spherical harmonic $f$. \\
$\O(n)$ & Elements of $\O$ of norm $n$. \\
            
\end{tabular}
\end{center}

\section{Acknowledgements} I thank foremostly my advisor Alex Gamburd for the insightful suggestion that I should look at this problem and for his support throughout. I also thank Martin Weissman for conversations about automorphic representations which have been invaluable in my writing of this paper. 

%%%%%%%%%%%%%%%%%%%%%%%%%%%%%%%%%%%%%%%%%%%%%%%%%%%%%%%%%%%%%%%%%%%%%%%%%%%%%%%%%%%%%%%%%%%%%%%%%%%
%%%%%%%%%%%%%%%%%%%%%%%%%%%%%%%%%%%%%%%%%%%%% BIBLIOGRAPHY %%%%%%%%%%%%%%%%%%%%%%%%%%%%%%%%%%%%%%%%
%%%%%%%%%%%%%%%%%%%%%%%%%%%%%%%%%%%%%%%%%%%%% BIBLIOGRAPHY %%%%%%%%%%%%%%%%%%%%%%%%%%%%%%%%%%%%%%%%
%%%%%%%%%%%%%%%%%%%%%%%%%%%%%%%%%%%%%%%%%%%%%%%%%%%%%%%%%%%%%%%%%%%%%%%%%%%%%%%%%%%%%%%%%%%%%%%%%%%

\end{document}